\documentclass[a4paper,12pt,reqno]{amsart}
\usepackage[T1]{fontenc}
\usepackage[utf8]{inputenc}
\usepackage{bera}
\usepackage{amssymb,dsfont,mathrsfs,mathabx}
\usepackage[margin=1in]{geometry}
\usepackage{enumitem}
\usepackage{xparse}
\usepackage{pdflscape,afterpage}
\usepackage{xcolor}
\usepackage[bookmarksdepth=2]{hyperref}

\numberwithin{equation}{section}

\newtheorem{theorem}{Theorem}[section]
\newtheorem{lemma}[theorem]{Lemma}
\newtheorem{corollary}[theorem]{Corollary}
\newtheorem{proposition}[theorem]{Proposition}

\theoremstyle{definition}
\newtheorem{definition}[theorem]{Definition}

\newtheorem{example}[theorem]{Example}
\newtheorem{remark}[theorem]{Remark}

\DeclareMathOperator{\ind}{\mathds{1}}

\newcommand{\test}{\mathscr{D}}
\newcommand{\schw}{\mathscr{S}}
\newcommand{\fourier}{\mathscr{F}}
\newcommand{\schwconv}{\circledast}
\newcommand{\schwprod}{\odot}
\newcommand{\conv}{\boxasterisk}
\newcommand{\profile}{Y}
\newcommand{\pr}{\mathds{P}}
\newcommand{\ex}{\mathds{E}}

\newcommand{\G}{\mathds{G}}
\newcommand{\R}{\mathds{R}}
\newcommand{\Z}{\mathds{Z}}

\newcommand{\ph}{\varphi}
\newcommand{\eps}{\varepsilon}

\renewcommand{\le}{\leqslant}
\renewcommand{\ge}{\geqslant}

\newcommand{\dual}[1]{{#1}^\dagger}

\NewDocumentCommand{\formula}{ssom}{%
 \IfBooleanTF{#1}{%
  \IfBooleanTF{#2}{%
   \IfValueTF{#3}%
    {\begin{align}\label{#3}\begin{gathered}#4\end{gathered}\end{align}}%
    {\begin{align}#4\end{align}}%
  }{%
   \IfValueTF{#3}%
    {\begin{align}\label{#3}\begin{aligned}#4\end{aligned}\end{align}}%
    {\begin{gather*}#4\end{gather*}}%
  }%
 }{%
  \IfValueTF{#3}%
   {\begin{align}\label{#3}#4\end{align}}%
   {\begin{align*}#4\end{align*}}%
 }%
}

\newcommand{\ignore}[1]{}

\begin{document}

\title{Liouville's theorems for Lévy operators}
\author{Tomasz Grzywny, Mateusz Kwaśnicki}
\thanks{Work supported by the Polish National Science Centre (NCN) grants no.\@ 2017/27/B/ST1/01339~(TG) and 2019/33/B/ST1/03098~(MK)}
\address{Tomasz Grzywny, Mateusz Kwaśnicki \\ Faculty of Pure and Applied Mathematics \\ Wrocław University of Science and Technology \\ ul. Wybrzeże Wyspiańskiego 27 \\ 50-370 Wrocław, Poland}
\email{tomasz.grzywny@pwr.edu.pl,mateusz.kwasnicki@pwr.edu.pl}
\date{\today}

\begin{abstract}
Let $L$ be a Lévy operator. A function $h$ is said to be harmonic with respect to $L$ if $L h = 0$ in an appropriate sense. We prove Liouville's theorem for positive functions harmonic with respect to a general Lévy operator $L$: such functions are necessarily mixtures of exponentials. For signed harmonic functions we provide a fairly general result, which encompasses and extends all Liouville-type theorems previously known in this context, and which allows to trade regularity assumptions on $L$ for growth restrictions on $h$. Finally, we construct an explicit counterexample which shows that Liouville's theorem for signed functions harmonic with respect to a general Lévy operator $L$ does not hold.
\end{abstract}

\maketitle

%
%                            ---------- o ----------
%

\section{Introduction}
\label{sec:intro}

%                            ---------- o ----------

\subsection{Liouville's theorem and Lévy operators}
\label{sec:intro:levy}

A classical result due to Liouville and Cauchy, traditionally called \emph{Liouville's theorem}, states that every bounded harmonic function in $\R^d$ is constant. This was extended by Bôcher and Picard, who proved that a one-sided bound is sufficient: every positive harmonic function in $\R^d$ is constant. A yet another variant of Liouville's theorem asserts that every harmonic function bounded by a polynomial (and again a one-sided bound is sufficient) is in fact a harmonic polynomial.

Liouville's theorem has been extended in various directions. Here we study variants of Liouville's theorem for \emph{Lévy operators}, that is, operators $L$ of the form
\formula*[eq:levy:operator]{
 L f(x) & = \sum_{j, k = 1}^d a_{jk} \partial_{jk} f(x) + \sum_{j = 1}^d b_j \partial_j f(x) \\
 & \qquad + \int_{\R^d \setminus \{0\}} \biggl(f(x + y) - f(x) - \ind_B(y) \sum_{j = 1}^d y_j \partial_j f(x)\biggr) \nu(dy) ,
}
acting on an appropriate class of functions on $\R^d$. Here $B$ is the unit ball, $(a_{jk}) \in \R^{d \times d}$ is a nonnegative definite symmetric matrix, $(b_j) \in \R^d$ is a vector, and $\nu$ is a nonnegative measure on $\R^d \setminus \{0\}$ such that $\int_{\R^d \setminus \{0\}} \min\{1, |y|^2\} \nu(dy) < \infty$, the so-called \emph{Lévy measure}.

Various equivalent descriptions of the class of Lévy operators are available. In particular, the following conditions are equivalent: $L$ is a Lévy operator; $L$ is the generator of a \emph{Lévy process} $X_t$; $L$ generates a strongly continuous semigroup of translation invariant Markov operators; $-L$ is a Fourier multiplier whose symbol $\Psi$ is a continuous negative definite function vanishing at the origin; $L$ is a translation invariant integro-differential operator vanishing on constants and satisfying the maximum principle. We remark that the function $\Psi$ mentioned above is the characteristic exponent of the Lévy process $X_t$, given by the Lévy--Khintchine formula
\formula[eq:lk]{
 \Psi(\xi) & = a(\xi, \xi) - i b \xi + \int_{\R^d \setminus \{0\}} \bigl(1 - e^{i \xi y} + i \xi y \ind_B(y)\bigr) \nu(dy) ,
}
and for every $t \in [0, \infty)$, the characteristic function of the random variable $X_t$ is equal to $e^{-t \Psi}$. Here $a(\xi, \xi) = \sum_{j, k = 1}^d a_{jk} \xi_j \xi_k$, and $b \xi = \sum_{j = 1}^k b_j \xi_j$.

The Laplace operator $\Delta$ is a prime example of a Lévy operator, and a smooth function $f$ is harmonic in $\R^d$ if and only if $\Delta f = 0$ in $\R^d$. In a similar way, given a general Lévy operator $L$, a smooth function $f$ is said to be harmonic with respect to $L$, or \emph{$L$-harmonic}, if $L f = 0$ in $\R^d$. In probabilistic terms, the Lévy operator $L$ is the generator of a Lévy process $X_t$, and $f$ is $L$-harmonic if and only if $f(X_t)$ is a local martingale.

The question that we address in this article is: which variants of Liouville's theorem remain true for $L$-harmonic functions, where $L$ is a given Lévy operator? This problem was tackled by several authors over the last few years, and in fact its history can be traced back to the seminal work of M.~Riesz~\cite{riesz} on functions harmonic with respect to the fractional Laplace operator $-(-\Delta)^s$, another widely studied Lévy operator. Indeed: Liouville's theorem for bounded harmonic functions follows immediately from Harnack's inequality, and an appropriate variant of the latter result for the fractional Laplace operator is given in~\cite{riesz}.

The goal of this paper is three-fold. First, we prove a general Liouville's theorem for nonnegative $L$-harmonic functions: every such function is a mixture of $L$-harmonic exponentials; see Section~\ref{sec:intro:positive}. Next, we provide a general framework for proving Liouville's theorems for signed polynomially bounded $L$-harmonic functions by means of Fourier transformation. Under appropriate additional assumptions, we prove that all signed $L$-harmonic functions which correspond to tempered distributions have their spectrum (that is, the support of the Fourier transform) contained in the zero set of the characteristic exponent of $L$; see Section~\ref{sec:intro:signed}. Finally, we construct a surprising example which proves that one cannot completely drop the additional assumptions mentioned above: Liouville's theorem for signed, polynomially bounded $L$-harmonic functions does not hold in full generality; see Section~\ref{sec:intro:counter}.

In the introduction, we state simplified variants of our main results, and we provide references to the corresponding full statements given later in the text. We commonly impose the following standard assumption.

\begin{definition}
\label{def:subgroup}
A Lévy operator $L$ is said \emph{not to be concentrated on a proper closed subgroup of $\R^d$} if the Fourier symbol $\Psi$ of $-L$ (that is, the characteristic exponent of the corresponding Lévy process $X_t$) is equal to zero only at the origin.
\end{definition}

The above condition holds if and only if the closed subgroup $\G$ of $\R^d$ generated by the union of the supports of the distributions of $X_t - X_0$ is equal to $\R^d$. The subgroup $\G$ can be described as the smallest closed subgroup which contains: (a)~eigenspaces of the matrix $(a_{jk})$ which correspond to nonzero eigenvalues; (b)~the support of $\nu$; and (c)~an appropriately defined drift vector. This result is due to Tortrat~\cite{tortrat}; we refer to Théorème~3 in~\cite{simon} and to~\cite{adej} for details, as well as to Section~24 in~\cite{sato} (in particular, Definitions~24.13 and~24.21 and Proposition~24.14 therein) for further discussion. We also note that other names, such as \emph{nonlattice condition}, are commonly used for the condition described in Definition~\ref{def:subgroup}.

We remark that if a Lévy operator $L$ is concentrated on a proper closed subgroup $\G$ of $\R^d$, then $L$ acts on each coset $x + \G$ separately. In particular, every $\G$-periodic function on $\R^d$ (that is, a function $f$ such that $f(x + z) = f(x)$ for every $x \in \R^d$ and $z \in \G$) is $L$-harmonic. The proper way to state Liouville's theorem for a Lévy operator $L$ concentrated on $\G$ is to consider functions defined only on $\G$. Although we do not discuss this extension in the introduction, our main results carry over to this situation: for positive $L$-harmonic functions we state this explicitly in Section~\ref{sec:positive}, while for signed $L$-harmonic function this extension corresponds to Fourier symbols $\Psi$ with zero sets other than $\{0\}$, which are allowed in Section~\ref{sec:signed}.

%                            ---------- o ----------

\subsection{Positive $L$-harmonic functions}
\label{sec:intro:positive}

In order to state Liouville's theorem for positive $L$-harmonic functions, we need an auxiliary definition.

If $L$ is a Lévy operator, we denote by $\dual{L}$ the Lévy operator \emph{dual} to $L$, obtained by conjugation of $L$ with the reflection $x \mapsto -x$. More precisely, if $L$ is given by~\eqref{eq:levy:operator}, then $\dual{L}$ has a similar representation with the same matrix $(\dual{a}_{jk}) = (a_{jk})$, with vector $(\dual{b}_j) = (-b_j)$, and with Lévy measure $\dual{\nu}(A) = \nu(-A)$. We remark that $\dual{L}$ is an operator formally adjoint to the operator $L$, and if both are considered as unbounded operators on $L^2(\R^d)$, then $L$ and $\dual{L}$ are indeed adjoint operators.

\begin{definition}
\label{def:harmonic}
A locally integrable function $h$ is said to be \emph{$L$-harmonic} (in the weak sense) if
\formula{
 \int_{\R^d} h(x) \dual{L} \ph(x) dx & = 0
}
for every smooth, compactly supported function $\ph$. Here we implicitly assume absolute convergence of the integral in the left-hand side.
\end{definition}

In fact, we will allow for a slightly more general definition; see Lemma~\ref{lem:harmonic} in Section~\ref{sec:pre:levy} for a detailed discussion. Of course, whenever we say that an $L$-harmonic function $h$ in the above sense is constant or equal to a polynomial, in fact we mean equality almost everywhere with respect to the Lebesgue measure.

Recall that $B$ is the unit ball in $\R^d$. It is relatively simple to see that if $h$ has continuous second-order partial derivatives,
\formula{
 & \int_{\R^d \setminus B} |h(x + y)| \nu(dy)
}
is a locally integrable function of $x \in \R^d$, and $L h = 0$ in $\R^d$ in the pointwise sense, then $h$ is $L$-harmonic in the weak sense.

The following statement, which builds upon and extends Theorem~5.7 in~\cite{bs}, is our first main result.

\begin{theorem}[see Theorem~\ref{thm:positive:liouville}]
\label{thm:positive:liouville:0}
Let $L$ be a Lévy operator which is not concentrated on a proper closed subgroup of $\R^d$, and let $h$ be a nonnegative function on $\R^d$ which satisfies $L h = 0$ in the weak sense. Then
\formula{
 h(x) & = \int_\Lambda e^{\lambda x} \mu(d\lambda)
}
for a unique nonnegative measure $\mu$ on the set $\Lambda$ of those vectors $\lambda \in \R^d$ for which the function $e_\lambda(x) = e^{\lambda x}$ satisfies $L e_\lambda(x) = 0$.
\end{theorem}

The above result follows directly from the more general Theorem~\ref{thm:positive:liouville}, which allows $h$ to be an arbitrary Schwartz distribution such that Definition~\ref{def:harmonic} makes sense, and $L$ to be concentrated on a proper closed subgroup $\G$ of $\R^d$.

Theorem~\ref{thm:positive:liouville:0} for Lévy operators which generate random walks was proved by Deny in~\cite{deny} (see Théorème~3 therein). In terms of the representation~\eqref{eq:levy:operator}, this result corresponds to $a_{jk} = 0$, $b_j = 0$ and $\nu$ a finite measure. We stress that this special case is a crucial ingredient of our proof. A variant of Theorem~\ref{thm:positive:liouville:0} was proved by Berger and Schilling in~\cite{bs} (see Theorem~5.7 therein) with the additional assumption that $h$ is bounded by a submultiplicative function integrable with respect to the Lévy measure $\nu$ over the complement of the unit ball. An alternative proof was given recently by Berger, Schilling and Shargorodsky in~\cite{bss} (see Theorem~13 therein).

Since every bounded mixture of exponentials is constant, Theorem~\ref{thm:positive:liouville:0} in particular implies that bounded $L$-harmonic functions are necessarily constant, provided that $L$ is not concentrated on a proper closed subgroup of $\R^d$. This variant of Theorem~\ref{thm:positive:liouville:0} has some history. For the fractional Laplace operator, this result follows directly from Harnack's inequality proved by M.~Riesz in~\cite{riesz} (see formula~(5) in Chapter~V therein); an essentially equivalent argument was given by Bogdan, Kulczycki and Nowak in~\cite{bkn} (see Lemma~3.2 therein). Similar Liouville's theorem for Lévy operators with measure $\nu(dy)$ decaying exponentially fast at infinity and comparable with $|y|^{-d - 1} dy$ near the origin was given by Barlow, Bass and Gui in~\cite{bbg} (see Theorem~1.17 therein), while the result for rather general Lévy operators with nondegenerate second order local term was proved by Priola and Zabczyk in~\cite{pz} (see Theorem~3.8 therein). The case of general Lévy operators $L$ was solved completely by Alibaud, del Teso, Endal and Jakobsen in~\cite{adej} (see Theorems~1.1 and~1.2 therein) using analytical methods. We remark that the probabilistic argument applied by Berger and Schilling in~\cite{bs} (see Theorem~4.4 therein for the statement for bounded $L$-harmonic functions) is quite different and more general.

It is worthy of mention that from the potential-theoretic perspective, proving a Liouville's theorem for positive $L$-harmonic functions is equivalent to the identification of \emph{minimal} $L$-harmonic functions in $\R^d$. The latter problem is meaningful for a much wider class of linear operators $L$, and by Choquet's theory, positive $L$-harmonic functions are always integral mixtures of minimal $L$-harmonic functions. This is known as \emph{Martin representation}, after a very influential work of Martin~\cite{martin} for classical harmonic functions. Various results of this type are available, and we mention~\cite{ar,bkk,ksv-1,ksv-2,ksv-3,rw} as a sample of references dealing with Lévy operators in unbounded domains.

%                            ---------- o ----------

\subsection{Signed $L$-harmonic functions}
\label{sec:intro:signed}

In order to study signed $L$-harmonic functions, we apply Fourier transform methods. These are limited to the class of tempered distributions, and for this reason we need to modify slightly the definition of an $L$-harmonic function. We denote by $\schw$ the Schwartz class of rapidly decreasing smooth functions on $\R^d$, and by $\schw'$ the class of tempered distributions.

\begin{definition}
\label{def:harmonic:schw}
A locally integrable function $h$ with at most polynomial growth at infinity, or, more generally, a tempered distribution $h$, is said to be \emph{$L$-harmonic} (in the sense of tempered distributions) if
\formula{
 \int_{\R^d} h * \psi(x) \dual{L} \ph(x) dx & = 0
}
for every $\ph, \psi \in \schw$. Here we implicitly assume absolute convergence of the integral in the left-hand side.
\end{definition}

We remark that if $h$ is an $L$-harmonic function in the weak sense (according to Definition~\ref{def:harmonic}, and if additionally
\formula{
 & |h(x)| + \int_{\R^d \setminus B} |h(x + y)| \nu(dy)
}
is bounded by a polynomial (as a function of $x \in \R^d$), then $h$ is $L$-harmonic in the sense of tempered distributions (according to Definition~\ref{def:harmonic:schw}); see Lemma~\ref{lem:harmonic:schw} in Section~\ref{sec:pre:levy}.

A Lévy operator $L$ is a convolution operator. The corresponding convolution kernel is an appropriate tempered distribution, which we denote by $\check{L}$. The distributional Fourier transform $\fourier \check{L}$ of the tempered distribution $\check{L}$ is a continuous function: it is equal to $-\Psi$, the Fourier symbol of $L$. Recall that, in probabilistic terms, $\Psi$ is the characteristic exponent of the corresponding Lévy process $X_t$.

For simplicity, below we assume that $L$ is not concentrated on a proper closed subgroup of $\R^d$, that is, $\Psi(\xi) \ne 0$ for $\xi \ne 0$. We stress, however, that the most general version of our result stated in Theorem~\ref{thm:liouville} also covers general Lévy operators, as well as other convolution operators, as long as their Fourier symbols are continuous functions in an appropriate class.

Before we state our main result in this section, Theorem~\ref{thm:liouville:0}, it is instructive to describe briefly our argument. Suppose that $h$ is an $L$-harmonic function in the sense of tempered distributions. Our proof of Liouville's theorem for signed $L$-harmonic functions consists of the following steps.
\begin{enumerate}[label=(\Roman*)]
\item\label{it:step:a} Observe that $\check{L}$ and $h$ are \emph{convolvable} (as tempered distributions), and
\formula{
 \check{L} * h & = L h = 0 .
}
\item\label{it:step:b} Use the \emph{Fourier exchange formula} to find out that the \emph{muliplicative product} of Fourier transforms of $\check{L}$ and $h$ is well-defined (in the sense of tempered distributions), and
\formula{
 \fourier \check{L} \cdot \fourier h & = \fourier (\check{L} * h) = 0 .
}
\item\label{it:step:c} Use an appropriate variant of \emph{Wiener's $1/f$ theorem} to deduce that $\fourier h = 0$ on $\R^d \setminus \{0\}$.
\item\label{it:step:d} Conclude that $h$ is a polynomial
\end{enumerate}
Steps~\ref{it:step:a} and~\ref{it:step:d} present no difficulties, while step~\ref{it:step:b} is a known result in the theory of tempered distributions; we refer to Sections~\ref{sec:pre:distr} and~\ref{sec:pre:levy} for a detailed discussion. The only problematic part in the above argument is step~\ref{it:step:c}.

We remark that if $\Psi = -\fourier \check{L}$ is a smooth function in $\R^d \setminus \{0\}$, then also step~\ref{it:step:c} is completely standard, and in this way we recover the variant of Liouville's theorem given in Theorem~3.2 in~\cite{bs}.

If we assume that $h$ is a bounded function, then the usual Wiener's $1/f$ theorem can be easily adapted to make step~\ref{it:step:c} rigorous. This leads to a significantly shorter proof of Liouville's theorem originally given in Theorem~1.1 in~\cite{adej} and, independently, in Theorem~4.4 in~\cite{bs}.

The two examples discussed above are in some sense extreme cases. We also provide intermediate variants. However, we need to keep balance between smoothness conditions on $\Psi$ and growth restrictions of $h$. Detailed statements are given in Corollary~\ref{cor:liouville:0} below.

A rigorous statement of our general result requires the following two auxiliary definitions.

\begin{definition}
\label{def:algebra}
We say that an algebra $W$ of continuous functions on $\R^d$ is a \emph{Wiener-type algebra} if:
\begin{enumerate}[label=(\alph*)]
\item\label{it:al:1} every element of $W$ is a tempered distribution;
\item\label{it:al:2} if $\Phi \in W$ and $\ph \in \schw$, then $\ph \cdot \Phi \in W$;
\item\label{it:al:3} if $K \subseteq \R^d$ is a compact set, $\Phi \in W$ and $\Phi(\xi) \ne 0$ for every $\xi \in K$, then there is $\tilde{\Phi} \in W$ such that $\Phi(\xi) \tilde{\Phi}(\xi) = 1$ for every $\xi \in K$.
\end{enumerate}
We say that a tempered distribution $\Psi$ belongs to $W$ \emph{locally} on an open set $U$ if for every compact set $K \subseteq U$ there is a tempered distribution $\Phi \in W$ such that $\Psi = \Phi$ in a neighbourhood of $K$.
\end{definition}

We note that the Schwartz class $\schw$, or the class of Fourier transforms of integrable functions (that is, the usual Wiener algebra), are Wiener-type algebras.

Clearly, conditions~\ref{it:al:1} and~\ref{it:al:2} are rather natural, and typically they are easy to check. The essential property of Wiener-type algebras is given in condition~\ref{it:al:3}, which can be thought of as a variant of Wiener's $1/f$ theorem.

\begin{definition}
\label{def:act}
We say that a tempered distribution $H$ \emph{acts} on a Wiener-type algebra $W$ if for every $\Phi, \Psi \in W$ we have the following identity of multiplicative products of tempered distributions:
\formula{
 (H \cdot \Phi) \cdot \Psi & = H \cdot (\Phi \cdot \Psi) .
}
In particular, we require that all products of tempered distributions in the above identity are well-defined.
\end{definition}

The following result apparently covers all known Liouville-type theorems on signed $L$-harmonic functions for Lévy operators $L$, and, together with Corollary~\ref{cor:liouville:0} below, it is our second main result.

\begin{theorem}[see Theorem~\ref{thm:liouville}]
\label{thm:liouville:0}
Let $W$ be a Wiener-type algebra of continuous functions on $\R^d$, and let $L$ be a Lévy operator which is not concentrated on a proper closed subgroup of $\R^d$. Suppose that the Fourier symbol $-\Psi$ of $L$ belongs to $W$ locally on $\R^d \setminus \{0\}$. Let $h$ be an $L$-harmonic function in the sense of tempered distributions, and suppose that $\fourier h$ acts on $W$. Then $f$ is a polynomial.
\end{theorem}

We remark that if $f$ is a polynomial and $L f$ is well-defined, then $L f$ is again a polynomial, and the degree of $L f$ is less than the degree of $f$. Thus, in order to find $L$-harmonic polynomials of degree $n$ it is sufficient to evaluate $L f$ for every monomial $f$ of degree at most $n$, and solve a system of linear equations. We also remark that if $L$ is isotropic (invariant under rotations), then $L$-harmonic polynomials $h$ are harmonic in the classical sense, that is, they satisfy $\Delta h = 0$.

The proof of Theorem~\ref{thm:liouville:0} is similar to the argument applied independently in~\cite{bss}. By specifying the Wiener-type algebra $W$, we obtain the following family of less abstract results. More variants of Liouville's theorem which follow from Theorem~\ref{thm:liouville:0} can be found in Section~\ref{sec:signed}.

\begin{corollary}
\label{cor:liouville:0}
Let $L$ be a Lévy operator which is not concentrated on a proper closed subgroup of $\R^d$, and let $h$ be an $L$-harmonic function in the sense of tempered distributions. Then $h$ is necessarily a polynomial if any of the following conditions holds:
\begin{enumerate}[label=(\alph*)]
\item\label{cor:liouville:a} the Fourier symbol $-\Psi$ of $L$ is smooth on $\R^d \setminus \{0\}$ (Corollary~\ref{cor:liouville:smooth});
\item\label{cor:liouville:b} $h$ is a bounded function (Corollary~\ref{cor:liouville:bounded});
\item\label{cor:liouville:c} for some $\alpha \ge 0$, the function $|x|^\alpha$ is integrable with respect to the Lévy measure $\nu$ on $\R^d \setminus B$, while $(1 + |x|)^{-\alpha} h(x)$ is a bounded function (Example~\ref{ex:liouville:oper:power});
\item\label{cor:liouville:d} for some $\beta \ge 0$, the function $(\log |x|)^\beta$ is integrable with respect to the Lévy measure $\nu$ on $\R^d \setminus B$, while $(\log(e + |x|))^{-\beta} h(x)$ is a bounded function (Example~\ref{ex:liouville:oper:log});
\item\label{cor:liouville:e} for some $\alpha > 0$, there is a constant $c$ such that the Lévy measure of the ball with centre $x$ and radius $1$ is bounded by $c |x|^{-d - \alpha}$ for every $x \in \R^d$ such that $|x| \ge 2$, while $(1 + |x|)^{-d - \alpha} h(x)$ is an integrable function (Example~\ref{ex:liouville:funct:power});
\item\label{cor:liouville:f} for some $\alpha > 0$ and $\beta \in \R$, or for $\alpha = 0$ and some $\beta > 1$, there is a constant $c$ such that the Lévy measure of the ball with centre $x$ and radius $1$ is bounded by $c |x|^{-d - \alpha} (\log |x|)^{-\beta}$ for every $x \in \R^d$ such that $|x| \ge 2$, while $(1 + |x|)^{-d - \alpha} (\log(e + |x|))^{-\beta} h(x)$ is an integrable function (Example~\ref{ex:liouville:funct:power:log}).
\end{enumerate}
\end{corollary}

As discussed after the statement of Theorem~\ref{thm:positive:liouville:0}, Liouville's theorem for bounded $L$-harmonic functions has been studied previously, and the general result stated in Corollary~\ref{cor:liouville:0}\ref{cor:liouville:a} is due to Alibaud, del~Teso, Endal and Jakobsen~\cite{adej} (see Theorems~1.1 and~1.2 therein). An independent proof was given by Berger and Schilling in~\cite{bs} (see Theorem~4.4 therein).

The variant for smooth symbols, given in Corollary~\ref{cor:liouville:0}\ref{cor:liouville:b}, follows easily from standard properties of tempered distributions and their Fourier transforms. Essentially the same statement is given by Berger and Schilling in~\cite{bs} (see Theorem~3.2 therein).

Liouville's theorem for $L$-harmonic functions under a moment condition on the Lévy measure similar to Corollary~\ref{cor:liouville:0}\ref{cor:liouville:c} was proved by Ros-Oton and Serra in~\cite{rs} (see Theorem~2.1 therein) for homogeneous Lévy-type operators (that is, generators of stable Lévy processes). General Lévy operators were considered by Kühn in~\cite{kuhn} (see Theorem~1 therein), and the present statement of Corollary~\ref{cor:liouville:0}\ref{cor:liouville:c} was independently found by Berger, Schilling and Shargorodsky in~\cite{bss} (see Theorem~8 therein). The same work contains a more general result (see Theorem~11 therein), equivalent to our Corollary~\ref{cor:liouville:oper}, which encompasses Corollary~\ref{cor:liouville:0}\ref{cor:liouville:c} and~\ref{cor:liouville:d}.

Corollary~\ref{cor:liouville:0}\ref{cor:liouville:e} for the fractional Laplace operator was proved Chen, D'Ambrosio and Li in~\cite{cdl} (see Theorem~1.3 therein) and by Fall in~\cite{fall} (see Theorem~1.1 therein). Extension similar to Corollary~\ref{cor:liouville:0}\ref{cor:liouville:e} was given by Fall and Weth in~\cite{fw} (see Theorem~1.4 therein). Corollary~\ref{cor:liouville:0}\ref{cor:liouville:f}, as well as more general Corollary~\ref{cor:liouville:funct}, seem to be new. In~\cite{dsv} (see Theorem~1.5 therein), Dipierro, Savin and Valdinoci proved a variant of Liouville's theorem for the fractional Laplace operator, which covers a wider class of harmonic functions; see Remark~\ref{rem:lizorkin} for further discussion.

We remark that if $\alpha > 0$ and the Lévy measure $\nu$ of the Lévy operator $L$ is comparable with $|y|^{-d - \alpha} dy$ when $|y|$ is large enough, then Corollary~\ref{cor:liouville:0}\ref{cor:liouville:e} provides a complete Liouville's theorem, with no restrictions on the $L$-harmonic function~$h$. Indeed: in this case $(1 + |x|)^{-d - \alpha} h(x)$ is automatically integrable whenever $h$ is $L$-harmonic, and additionally the two notions of $L$-harmonicity (Definitions~\ref{def:harmonic} and~\ref{def:harmonic:schw}) are equivalent. The same remark applies to Lévy operators with Lévy measure comparable with $|y|^{-d - \alpha} (\log |y|)^{-\beta} dy$ when $|y|$ is large enough, where $\alpha$ and $\beta$ are as in Corollary~\ref{cor:liouville:0}\ref{cor:liouville:f}.

%                            ---------- o ----------

\subsection{A counterexample}
\label{sec:intro:counter}

Although rather general, Theorem~\ref{thm:liouville:0} does not apply to arbitrary $L$-harmonic functions (in the sense of tempered distributions), and there is a reason for that: it turns out that it is not possible to make step~\ref{it:step:c} of our proof of Theorem~\ref{thm:liouville:0} rigorous without additional assumptions. In other words, even though $\Psi = -\fourier \check{L}$ is a positive, continuous function in $\R^d \setminus \{0\}$, there may exist a tempered distribution $\fourier h$ such that the multiplicative product of $\Psi \cdot \fourier h$ is well-defined in the sense of tempered distribution and $\Psi \cdot \fourier h = 0$, but nevertheless $\fourier h$ is nonzero in $\R^d \setminus \{0\}$. That is, in general, division by $\Psi$ turns out to be impossible.

A rigorous statement is given in the following theorem, which is our third main result.

\begin{theorem}
\label{thm:counter:0}
\begin{enumerate}[label=(\alph*)]
\item There is a bounded, positive, continuous function $f$ and a tempered distribution $g$ such that the $\schw'$-product $f \cdot g$ is well-defined and equal to zero even though $g$ is not identically zero (see Corollary~\ref{cor:counter:prod}).
\item There is a nontrivial probability measure $\mu$ on $\Z$ and a function $h$ on $\Z$ such that $\lim_{|x| \to \infty} (|x|^{-\eps} h(x)) = 0$ for every $\eps > 0$ and $h * \mu = \mu$, but $h$ is not constant, and hence not a polynomial (see Theorem~\ref{thm:counter:discrete}).
\item There is a Lévy operator $L$ and a smooth function $h$ such that $L$ is not concentrated on a proper closed subgroup of $\R^d$, for every $\eps > 0$ we have $\lim_{|x| \to \infty} (|x|^{-\eps} h(x)) = 0$, and $L h = 0$ (both pointwise and in the weak sense), but $h$ is not constant, and hence not a polynomial (see Theorem~\ref{thm:counter}).
\end{enumerate}
\end{theorem}

%                            ---------- o ----------

\subsection{Structure of the paper}
\label{sec:intro:structure}

The remaining part of the article is divided into four sections.

In Preliminaries we introduce the notation used in the remaining part of the paper (Section~\ref{sec:pre:general}), we discuss the notions of convolution and multiplication of distributions (Section~\ref{sec:pre:distr}), and connect them with Lévy operators (Section~\ref{sec:pre:levy}).

Theorem~\ref{thm:positive:liouville:0}, describing positive $L$-harmonic functions, is proved in Section~\ref{sec:positive}.

In Section~\ref{sec:counter}, we construct counterexamples to Liouville's theorem which prove Theorem~\ref{thm:counter:0}. We begin with operators on $\Z$ (Section~\ref{sec:counter:discrete}), then we reinterpret this example in terms of multiplication of distributions (Section~\ref{sec:counter:product}), and finally we deal with the case of operators on $\R^d$ (Section~\ref{sec:counter:continuous}).

The final part of the paper contains the proof of Theorem~\ref{thm:liouville:0} and Corollary~\ref{cor:liouville:0}, which provides various variants of Liouville's theorem for signed $L$-harmonic functions. We begin with the abstract result given in Theorem~\ref{thm:liouville:0} (Section~\ref{sec:signed:general}). Then, by choosing appropriate Wiener-type algebras, we show how this result leads to Liouville's theorems for operators with smooth symbols (Section~\ref{sec:signed:smooth}) and for bounded $L$-harmonic functions (Section~\ref{sec:signed:bounded}), which correspond to Corollary~\ref{cor:liouville:0}\ref{cor:liouville:a} and~\ref{cor:liouville:b}, respectively. In order to prove Corollary~\ref{cor:liouville:0}\ref{cor:liouville:c} and~\ref{cor:liouville:d}, we construct another Wiener-type algebra in Proposition~\ref{prop:oper} (Section~\ref{sec:signed:oper}). Similarly, Corollary~\ref{cor:liouville:0}\ref{cor:liouville:e} and~\ref{cor:liouville:f} is a consequence of Proposition~\ref{prop:funct} and auxiliary Lemma~\ref{lem:liouville:funct:radial} (Section~\ref{sec:signed:funct}).

\subsection*{Acknowledgements}

We thank Moritz Kaßmann for stimulating discussions about Liouville theorems, and the anonymous referee for helpful suggestions.

%
%                            ---------- o ----------
%

\section{Preliminaries}
\label{sec:pre}

%                            ---------- o ----------

\subsection{General notation}
\label{sec:pre:general}

We use $x, y, z \in \R^d$ for spatial variables, and $\xi, \eta \in \R^d$ for Fourier variables. If $\xi, x \in \R^d$, by $\xi x$ we denote the dot product of $\xi$ and $x$. We use the symbol $\delta_x(dy)$ for the Dirac delta measure at $x$.

We write $\test$ for the class of smooth, compactly supported functions on $\R^d$, and $\test'$ for the dual space, the class of \emph{Schwartz distributions}. As it is customary, if $f \in \test'$ and $\ph \in \test$, we write $\langle f, \ph \rangle$ for the value of $f$ at $\ph$.

We denote by $\schw$ the \emph{Schwartz class} of rapidly decreasing smooth functions on $\R^d$, and by $\schw'$ the dual space, the class of \emph{tempered distributions}. Again, if $f \in \schw'$ and $\ph \in \schw$, we write $\langle f, \ph \rangle$ for the value of $f$ at $\ph$. Note that $\test \subset \schw$ and $\schw' \subset \test'$.

The \emph{support} of a distribution $f$ is the smallest closed set $K$ such that $\langle f, \ph\rangle = 0$ whenever $\ph \in \test$ and $\ph(x) = 0$ for every $x$ in some neighbourhood of $K$.

We say that a distribution $f$ corresponds to a (locally integrable) function $\tilde{f}$ in an open set $U$ if $\langle f, \ph\rangle = \int_{\R^d} \tilde{f}(x) \ph(x) dx$ whenever $\ph \in \test$ has a compact support contained in $U$. In this case we do not distinguish between $f$ and $\tilde{f}$, and we use a single symbol for both the distribution $f$ and the function $\tilde{f}$. In a similar way, we identify (locally finite) measures with the corresponding distributions.

The \emph{Fourier transform} of $\ph \in \schw$ is defined by $\fourier \ph(\xi) = \int_{\R^d} e^{-i \xi x} f(x) dx$. The \emph{distributional Fourier transform} of a tempered distribution $f$ is a tempered distribution $\fourier f$, defined by the exchange formula $\langle \fourier f, \ph \rangle = \langle f, \fourier \ph \rangle$ for every $\ph \in \schw$. The inverse Fourier transform $\fourier^{-1}$ is defined in a similar way, with kernel $(2 \pi)^{-d} e^{i \xi x}$ rather than $e^{-i \xi x}$.

The \emph{spectrum} of a tempered distribution is the support of its distributional Fourier transform.

A (tempered) distribution $f$ is \emph{bounded} if the convolution $f * \ph$ (defined later in this section) is bounded for every $\ph \in \schw$. A bounded distribution extends to a continuous functional on the class of smooth functions with all derivatives integrable.

In a similar way, a (tempered) distribution $f$ is \emph{integrable} if $f * \ph$ is integrable for every $\ph \in \schw$. An integrable distribution extends to a continuous functional on the class of smooth functions with all derivatives bounded.

The Fourier transform of an integrable distribution $f$ coincides with a continuous function, defined by $\fourier f(\xi) = \langle f, e_\xi \rangle$, where $e_\xi(x) = e^{-i \xi x}$.

These and many other properties of Schwartz distributions and tempered distributions can be found in~\cite{vladimirov}.

%                            ---------- o ----------

\subsection{Crash course in convolution and multiplication of distributions}
\label{sec:pre:distr}

The convolution $f * \ph$ of $f \in \test'$ and $\ph \in \test$ is a smooth function defined by $f * \ph(x) = \langle f, \ph_x \rangle$, where $\ph_x(y) = \ph(x - y)$. The convolution of two Schwartz distributions $f$ and $g$ exists if there is a Schwartz distribution $f \conv g$ such that
\formula[eq:conv]{
 (f \conv g) * (\ph * \psi) & = (f * \ph) * (g * \psi)
}
for every $\ph, \psi \in \test$; in particular, we assume that the convolution of functions $f * \ph$ and $g * \psi$ in the right-hand side of~\eqref{eq:conv} exists in the usual sense. It is straightforward to check that if $\ph \in \test$, then $(f \conv g) * \ph = f \conv (g * \ph) = (f * \ph) \conv g$. Similarly, if $f, g, h \in \test'$, $f \conv g$ is well-defined and $h$ is compactly supported, then $(f \conv g) \conv h = f \conv (g \conv h)$ (and in particular all convolutions are well-defined). However, convolution of Schwartz distributions is not associative in general. For further information, we refer to Section~1 in~\cite{dv}.

In a similar way, the convolution $f * \ph$ of $f \in \schw'$ and $\ph \in \schw$ is a smooth function determined by $f * \ph(x) = \langle f, \ph_x \rangle$, where again $\ph_x(y) = \ph(x - y)$. The convolution of two tempered distributions $f$ and $g$ exists if there is a tempered distribution $f \schwconv g$ such that
\formula[eq:schwconv]{
 (f \schwconv g) * (\ph * \psi) & = (f * \ph) * (g * \psi)
}
for every $\ph, \psi \in \schw$; here, too, we assume that the middle convolution in the right-hand side of~\eqref{eq:schwconv} exists in the usual sense. The convolution $f \schwconv g$ is well-defined if, for example, $f, g \in \schw'$ and $f$ has compact support. It is again a simple exercise to check that if $\ph \in \schw$, then $(f \schwconv g) * \ph = f \schwconv (g * \ph) = (f * \ph) \schwconv g$, but the convolution of tempered distributions is not associative in general. We refer to Section~2 in~\cite{dv} for further details.

The two notions of the convolution of distributions: $\test'$-convolution $f \conv g$ and $\schw'$-convolution $f \schwconv g$, are clearly closely related to each other. However, we stress that there are tempered distributions $f, g \in \schw'$ such that $f \conv g$ is defined, while $f \schwconv g$ is not; see Section~3 in~\cite{dv} for a detailed discussion.

The $\schw'$-convolution of a bounded distribution $f$ and an integrable distribution $g$ is a bounded distribution, and the $\schw'$-convolution of two integrable distributions is again an integrable distribution. In fact, $\schw'$-convolution of any number of integrable distributions and a bounded distribution is commutative and associative. Obviously, every Schwartz class function corresponds to a distribution which is integrable and bounded. We refer to~\cite{dv} for further discussion.

The product $\psi \cdot f$ of $f \in \schw'$ and $\psi \in \schw$ is a tempered distribution defined by $\langle \psi \cdot f, \ph \rangle = \langle f, \ph \psi \rangle$ for every $\ph \in \schw$. The product of two tempered distributions $f$ and $g$ exists if there is a tempered distribution $f \schwprod g$ such that
\formula{
 \langle f \schwprod g, \ph \rangle & = \lim_{n \to \infty} \int_{\R^d} f * \alpha_n(x) g * \beta_n(x) \ph(x) dx
}
for every $\ph \in \schw$ and every sequences $\alpha_n$ and $\beta_n$ of nonnegative functions in $\test$ such that the integrals of $\alpha_n$ and $\beta_n$ are equal to $1$ and the supports of $\alpha_n$ and $\beta_n$ shrink to $\{0\}$ as $n \to \infty$. Multiplication of distributions extends multiplication of continuous functions, or a continuous function and a locally finite measure. Furthermore, multiplication of distributions is a local operation: if $f_1 = f_2$ and $g_1 = g_2$ in an open set $U$, then $f_1 \schwprod g_1 = f_2 \schwprod g_2$ in $U$ (provided that both products exist). Finally, if $\psi \in \schw$, then $\psi \cdot (f \schwprod g) = (\psi \cdot f) \schwprod g = f \schwprod (\psi \cdot g)$, but multiplication of tempered distributions is not associative in general. For a detailed discussion, we refer to~\cite{si}.

For every $f \in \schw'$ and $\ph \in \schw$, we have the exchange formula $\fourier(f * \ph) = (\fourier \ph) \cdot (\fourier f)$. As it was proved in~\cite{ho}, this result extends to the general case: if the convolution of tempered distributions $f$ and $g$ exists, then the product of their Fourier transforms is well-defined, and we have
\formula{
 \fourier(f \schwconv g) & = (\fourier f) \schwprod (\fourier g) .
}

%                            ---------- o ----------

\subsection{Lévy operators and distributions}
\label{sec:pre:levy}

Let $L$ be a Lévy operator, and let $\dual{L}$ be the dual operator, as in Definition~\ref{def:harmonic}. The map which assigns to $\ph \in \schw$ the value $\dual{L} \ph(0)$ is easily checked to be continuous, and hence it corresponds to some distribution, that we denote by $\check{L}$: $\dual{L} \ph(0) = \langle \check{L}, \ph \rangle$. Furthermore, if $\ph_x(y) = \ph(x - y)$, then $L \ph(x) = \dual{L} \ph_x(0) = \langle \check{L}, \ph_x \rangle$, that is, $L \ph(x) = \check{L} * \ph(x)$. Thus, $L$ is a convolution operator, with convolution kernel $\check{L} \in \schw'$. It is straightforward to see that $L \ph$ is integrable for every $\ph \in \schw$, and therefore $\check{L}$ is in fact an integrable distribution.

By a classical result in the theory of Lévy processes, the Fourier transform of $\check{L}$ is the characteristic exponent $\Psi$ defined in~\eqref{eq:lk}. In fact, this will serve us as the definition of the continuous function $\Psi$, and we will never need~\eqref{eq:lk}.

In Section~\ref{sec:positive} we work with a general, distributional definition of an $L$-harmonic function. Here we prove that this definition indeed extends Definition~\ref{def:harmonic}.

\begin{lemma}
\label{lem:harmonic}
Let $L$ be a Lévy operator. If a function $h$ is $L$-harmonic in the weak sense (according to Definition~\ref{def:harmonic}), then $h$ is $L$-harmonic in the sense of Schwartz distributions: $h$ corresponds to a Schwartz distribution which is convolvable with $\check{L}$, the convolution kernel of $L$, and we have $\check{L} \conv h = 0$.
\end{lemma}

\begin{proof}
By Definition~\ref{def:harmonic}, $h$ is a locally integrable function such that for every $\ph \in \test$ we have
\formula{
 \int_{\R^d} h(y) \dual{L} \ph(y) dy & = 0 .
}
Recall that $\langle \check{L}, \ph \rangle = \dual{L} \ph(0)$. Thus, if $\check{\ph}(x) = \ph(-x)$, then $\dual{L} \ph(x) = \check{L} * \check{\ph}(-x)$. It follows that
\formula{
 \int_{\R^d} h(y) \check{L} * \check{\ph}(-y) dy & = 0 ,
}
and in particular the integral is absolutely convergent. Replacing $\ph$ with $\ph_x(y) = \ph(x - y)$, we find that
\formula[eq:harmonic:aux]{
 \int_{\R^d} h(y) \check{L} * \ph(x - y) dx & = 0
}
for every $\ph \in \test$ and $x \in \R^d$. We will momentarily show that we can use Fubini's theorem to find that for every $\ph, \psi \in \test$ and $x \in \R^d$,
\formula*[eq:harmonic:double]{
 (h * \psi) * (\check{L} * \ph)(x) & = \int_{\R^d} h * \psi(x - y) \check{L} * \ph(y) dy \\
 & = \int_{\R^d} \biggl(\int_{\R^d} h(x - y - z) \psi(z) dz \biggr) \check{L} * \ph(y) dy \\
 & = \int_{\R^d} \biggl(\int_{\R^d} h(x - y - z) \check{L} * \ph(y) dy \biggr) \psi(z) dz \\
 & = \int_{\R^d} \biggl(\int_{\R^d} h(y) \check{L} * \ph(x - z - y) dy \biggr) \psi(z) dz = 0 .
}
Of course, this means that $\check{L} \conv h$ is well-defined and equal to zero, as desired.

In order to show absolute integrability of the double integral above, we denote by $C$ the supremum of $|\ph| + |\psi|$, and we choose $R$ large enough, so that $\ph(x) = \psi(x) = 0$ when $|x| \ge R$. Let $B_r(x)$ denote the ball of radius $r$ centred at $x$. By~\eqref{eq:levy:operator}, we have
\formula{
 |\check{L} * \ph(x)| & \le C \nu(B_R(-x))
}
whenever $|x| > R$. If $|x - y| > 2 R$ and $|z| < R$, then $|x - y - z| > R$. Thus,
\formula{
 \int_{\R^d} |\check{L} * \ph(x - z - y) \psi(z)| dz & \le C \int_{B_R(0)} |\check{L} * \ph(x - z - y)| dz \\
 & \le C^2 \int_{B_R(0)} \nu(B_R(y + z - x)) dz \le C^2 |B_R(0)| \nu(B_{2 R}(y - x)) .
}
On the other hand, if $\tilde{\ph} \in \test$ is chosen in such a way that $0 \le \tilde{\ph}(x) \le 1$ for all $x \in \R^d$, $\tilde{\ph}(x) = 1$ when $|x| \le 2 R$ and $\tilde{\ph}(x) = 0$ when $|x| \ge 3 R$, then, again by~\eqref{eq:levy:operator}, we have
\formula{
 |\check{L} * \tilde{\ph}(x)| & \ge \nu(B_{2 R}(-x))
}
whenever $|x| > 3 R$. It follows that if $|x - y| > 3 R$, then
\formula{
 \int_{\R^d} |\check{L} * \ph(x - z - y) \psi(z)| dz & \le C^2 |B_R(0)| |\check{L} * \tilde{\ph}(x - y)| .
}
Thus,
\formula{
 & \int_{\R^d \setminus B_{3 R}(x)} \int_{\R^d} |h(y) \check{L} * \ph(x - z - y) \psi(z)| dz dy \\
 & \qquad \le C^2 |B_R(0)| \int_{\R^d \setminus B_{3 R}(x)} |h(y) \check{L} * \tilde{\ph}(x - y)| dy ,
}
and the right-hand side is finite by~\eqref{eq:harmonic:aux}. Additionally, $h$ is integrable over $B_{3 R}(x)$, $\psi$ is integrable over $\R^d$, and $\check{L} * \ph$ is bounded, and hence
\formula{
 \int_{B_{3 R}(x)} \int_{\R^d} |h(y) \check{L} * \ph(x - z - y) \psi(z)| dz dy & < \infty .
}
Absolute integrability of the double integral in~\eqref{eq:harmonic:double} follows, and the proof is complete.
\end{proof}

Finally, we prove that under a natural growth condition, $L$-harmonic functions in the weak sense coincide with $L$-harmonic functions in the sense of tempered distributions. Note that with the notation introduced above, a function $h$ is $L$-harmonic in the sense of tempered distributions (according to Definition~\ref{def:harmonic:schw}) if and only if $\check{L} \schwconv h$ is well-defined and equal to $0$.

\begin{lemma}
\label{lem:harmonic:schw}
Let $L$ be a Lévy operator, let $\nu$ be the corresponding Lévy measure, and let $B$ denote the unit ball in $\R^d$. Let $h$ be a function which is $L$-harmonic in the weak sense (according to Definition~\ref{def:harmonic}), and such that
\formula{
 & |h(x)| + \int_{\R^d \setminus B} |h(x + y)| \nu(dy) ,
}
as a function of $x \in \R^d$, is bounded by a polynomial. Then $h$ is $L$-harmonic in the sense of tempered distributions (according to Definition~\ref{def:harmonic:schw}).
\end{lemma}

\begin{proof}
By Lemma~\ref{lem:harmonic}, we know that $\check{L} \conv h$ is well-defined and equal to zero, that is, for every $\ph, \psi \in \test$ the convolution of $\check{L} * \ph$ and $h * \psi$ is well-defined and equal to zero. Our goal is to prove a similar result for $\ph, \psi \in \schw$.

We choose a function $\ph_0 \in \test$ which takes values in $[0, 1]$, and such that $\ph_0(x) = 1$ when $|x| < 1$. Let $\check{L}_0 = \ph_0 \check{L}$ and $\check{L}_1 = \check{L} - \check{L}_0$. Then $\check{L}_0$ is a Schwartz distribution with compact support, and $\check{L}_1 = (1 - \ph_0) \nu$ is a nonnegative finite measure.

Since the support of $\check{L}_0$ is compact and $h$ defines a tempered distribution, the convolution $\check{L}_0 \schwconv h$ is well-defined. Furthermore, $\check{L}_1$ is a finite nonnegative measure and, by assumption, $\check{L}_1 * |h|$ is bounded by a polynomial. This implies that $\check{L}_1 \schwconv h$ is well-defined, too. Thus, $\check{L} \schwconv h$ is well-defined.

Clearly, $\check{L} \schwconv h = \check{L} \conv h$. However, by Lemma~\ref{lem:harmonic}, $\check{L} \conv h = 0$, and the proof is complete.
\end{proof}

%
%                            ---------- o ----------
%

\section{Positive harmonic functions}
\label{sec:positive}

In this section we prove Theorem~\ref{thm:positive:liouville:0}: we show that nonnegative $L$-harmonic functions are essentially mixtures of $L$-harmonic exponentials. The idea of the proof is very similar to the one used in~\cite{bs}: we reduce the problem to a convolution equation studied by Deny in~\cite{deny}. However, instead of the heat kernel (the distribution of the corresponding Lévy process at a fixed time) as in~\cite{bs}, we use the harmonic measure (the distribution of the Lévy process at the first exit time). This allows us to avoid integrability problems and thus remove unnecessary restrictions on the class of $L$-harmonic functions $h$.

For simplicity, Theorem~\ref{thm:positive:liouville:0} is stated for Lévy operators on $\R^d$ which are not concentrated on a proper closed subgroup of $\R^d$. In the general case, the smallest closed subgroup $\G$ of $\R^d$ on which $L$ is concentrated is isomorphic to $\R^{d - k} \times \Z^k$ for some $k \in \{0, 1, \ldots, d\}$. In order to cover this case, here we consider Lévy operators acting on $\G = \R^d \times \Z^k$ for arbitrary $d$ and $k$, and we continue to assume that $L$ is not concentrated on a proper subgroup of $\G$. We omit the obvious extensions of the notions discussed above for $\R^d$ to this more general case.

As mentioned above, the key tool in our proof is the main result (Théorème~3) of~\cite{deny}. The original statement allows $\G$ to be an arbitrary separable locally compact abelian group. We restrict our attention to $\G = \R^d \times \Z^k$. In this case the Haar measure on $\G$ is the product of the Lebesgue measure on $\R^d$ and the counting measure on $\Z^k$, and for simplicity we call it simply the Lebesgue measure on $\G$.

\begin{theorem}[Deny's theorem]
\label{thm:deny}
Let $\nu$ be a probability measure on $\G = \R^d \times \Z^k$ such that the closed group generated by the support of $\nu$ is equal to $\G$, and let $h$ be a locally finite nonnegative measure on $\G$ such that $h * \nu = h$. Then $h$ is absolutely continuous with respect to the Lebesgue measure on $\G$, and the density function is equal to
\formula{
 h(x) & = \int_\Lambda e^{\lambda x} \mu(d\lambda)
}
for a unique nonnegative measure $\mu$ on the set $\Lambda$ of those vectors $\lambda \in \R^d \times \R^k$ for which the function $e_\lambda(x) = e^{\lambda x}$ satisfies $e_\lambda * \nu = e_\lambda$.
\end{theorem}

Although we avoid as much as possible probabilistic tools, in this section some well-known results from the theory of Lévy processes play an important role. Namely, in the next paragraph we introduce two standard potential-theoretic objects: the Green kernel and the harmonic measure, and in the proof of our Liouville's theorem we use the probabilistic definition of the notion of $L$ not being concentrated on a proper closed subgroup of $\G$.

For every bounded open set $D \subset \G$ there are associated \emph{Green kernel} $G_D(x, dy)$ and the \emph{harmonic measure} $H_D(x, dz)$, with the following properties. If $x \in D$, then $G_D(x, dy)$ is a finite, nonnegative measure with respect to $y$, concentrated on $D$, while $H_D(x, dz)$ is a probability measure with respect to $z$, concentrated on $\G \setminus D$. If $x \notin D$, then $G_D(x, dy)$ is a zero measure, while $H_D(x, dz) = \delta_x(dz)$. Finally, for every $\ph \in \test$ and $x \in \G$ we have
\formula[eq:iw]{
 \ph(x) & = \int_{\G} \ph(z) H_D(x, dz) - \int_{\G} L \ph(y) G_D(x, dy) .
}
While the above objects can be constructed analytically, their probabilistic description is much simpler and far more intuitive. Let $X_t$ be the Lévy process with generator $L$, and let $\pr^x$ and $\ex^x$ denote the probability and expectation corresponding to the process $X_t$ started at $x \in \G$. Let $\tau_D$ denote the first exit time from $D$:
\formula{
 \tau_D & = \inf\{ t \in [0, \infty) : X_t \notin D \} .
}
Then $H_D(x, dz)$ is the distribution of $X_t$ at the first exit time $\tau_D$, and $G_D(x, dy)$ is the mean occupation measure up to time $\tau_D$:
\formula{
 H_D(x, A) & = \pr^x(X_{\tau_D} \in A) , \\
 G_D(x, A) & = \ex^x \int_0^{\tau_D} \ind_A(X_t) dt .
}
Formula~\eqref{eq:iw} is known as \emph{Dynkin's formula}, and it is one of the fundamental tools in the study of Markov processes; see Theorem~5.1 in~\cite{dynkin}.

For a bounded open set $D$ which contains the origin, we denote
\formula{
 \check{G}_D(A) & = G_D(0, -A) && \text{and} & \check{H}_D(A) & = H_D(0, -A) .
}
Recall that by $\check{L}$ we denote the convolution kernel of the Lévy operator $L$. Then~\eqref{eq:iw} implies that for every $\ph \in \test$ we have
\formula{
 \ph(0) & = \int_{\G} \ph(-z) \check{H}_D(dz) - \int_{\G} L \ph(-y) \check{G}_D(dy) \\
 & = \check{H}_D * \ph(0) - (\check{L} * \ph) * \check{G}_D(0) .
}
Since $\check{G}_D$ is a finite measure with compact support, the convolution $\check{L} \conv \check{G}_D$ is well-defined, and $(\check{L} * \ph) * \check{G}_D = (\check{L} * \ph) \conv \check{G}_D = \ph * (\check{L} \conv \check{G}_D)$. Thus,
\formula{
 \ph(0) & = \ph * \bigl(\check{H}_D - \check{L} \conv \check{G}_D\bigr)(0)
}
for every $\ph \in \test$. But this is another way to say that
\formula[eq:iw2]{
 \delta_0 & = \check{H}_D - \check{L} \conv \check{G}_D .
}
We will use the above identity in the proof of the following theorem, which is the main result of this section.

\begin{theorem}[Liouville's theorem for positive solutions]
\label{thm:positive:liouville}
Let $L$ be a Lévy operator on $\G = \R^d \times \Z^k$, which is not concentrated on a proper closed subgroup of $\G$, and let $h$ be a locally finite nonnegative measure on $\G$ which is $L$-harmonic in the weak sense, or, more generally, in the sense of Schwartz distributions (see Lemma~\ref{lem:harmonic}). Then $h$ is absolutely continuous with respect to the Lebesgue measure on $\G$, and the density function is equal to
\formula[eq:positive:liouville]{
 h(x) & = \int_\Lambda e^{\lambda x} \mu(d\lambda)
}
for a unique nonnegative measure $\mu$ on the set $\Lambda$ of those vectors $\lambda \in \R^d \times \R^k$ for which the function $e_\lambda(x) = e^{\lambda x}$ satisfies $L e_\lambda(x) = 0$.
\end{theorem}

We remark that when we write $L e_\lambda(x) = 0$ above, we mean that $L e_\lambda(x)$ is well defined pointwise, according to the definition~\eqref{eq:levy:operator}. However, as we will see in the proof of Theorem~\ref{thm:positive:liouville}, this is equivalent to $e_\lambda$ being $L$-harmonic in the weak sense. Additionally, since $L e_\lambda(x) = e_\lambda(x) L e_\lambda(0)$, we have $L e_\lambda(x) = 0$ for every $x \in \G$ whenever $L e_\lambda(x) = 0$ for a single $x \in \G$.

\begin{proof}
\emph{Step 1.} Suppose that the convolution $\check{L} \conv h$ is well-defined and equal to $0$. We consider a bounded open set $D$ such that $0 \in D$, and we use the notation introduced above. Recall that the convolution of distributions is associative when one of the factors is compactly supported. Thus,
\formula{
 0 & = (h \conv \check{L}) \conv \check{G}_D = h \conv (\check{L} \conv \check{G}_D) .
}
By~\eqref{eq:iw2}, we have $\check{L} \conv \check{G}_D = \check{H}_D - \delta_0$, and so $h \conv (\check{H}_D - \delta_0)$ is well-defined and equal to zero. Clearly, $h \conv \delta_0 = h$ is well-defined, and so
\formula{
 h & = 0 + h \conv \delta_0 = h \conv (\check{H}_D - \delta_0) + h \conv \delta_0 = h \conv \check{H}_D .
}
Since $h$ and $\check{H}_D$ are nonnegative measures, we have $h \conv \check{H}_D = h * \check{H}_D$. Thus,
\formula{
 h * \check{H}_D & = h .
}

\emph{Step 2.} The desired result essentially follows now from Deny's theorem (Theorem~\ref{thm:deny}): if the support of $\check{H}_D$ generates a dense subgroup of $\G$, then the equality $h * \check{H}_D = h$ implies the desired representation~\eqref{eq:positive:liouville} of $h$, with $\Lambda$ replaced by the set $\Lambda_D$ of those vectors $\lambda \in \R^d$ for which the function $e_\lambda(x) = e^{\lambda x}$ satisfies $e_\lambda * \check{H}_D = e_\lambda$.

A detailed argument, however, requires some care: we need to show that $\Lambda_D = \Lambda$, and we need to handle the case when the support of $\check{H}_D$ is contained in a proper closed subgroup of $\G$.

\emph{Step 3.} We first show that $\Lambda_D = \Lambda$. In fact, we prove a stronger statement: if $e_\lambda * \check{H}_D$ is not everywhere infinite, then $e_\lambda \conv \check{L}$ is well-defined, and the sign of $e_\lambda \conv \check{L}$ is the same as the sign of $e_\lambda * \check{H}_D - e_\lambda$.

Denote $\alpha_{\lambda, D} = e_\lambda * \check{H}_D(0)$. Observe that
\formula{
 e_\lambda * \check{H}_D(x) & = \int_{\G} e_\lambda(x - y) \check{H}_D(dy) = e_\lambda(x) \int_{\G} e_\lambda(-y) \check{H}_D(dy) = \alpha_{\lambda, D} e_\lambda(x) .
}
In particular, since $e_\lambda * \check{H}_D$ is not everywhere infinite, $\alpha_{\lambda, D}$ is finite. We find that
\formula{
 (\alpha_{\lambda, D} - 1) e_\lambda & = e_\lambda * (\check{H}_D - \delta_0) = e_\lambda \conv (\check{H}_D - \delta_0) = e_\lambda \conv (\check{L} \conv \check{G}_D) .
}
Suppose that $\ph \in \test$ is a nonnegative function which is not identically equal to zero. We have
\formula{
 (\alpha_{\lambda, D} - 1) e_\lambda * \ph & = (e_\lambda \conv (\check{L} \conv \check{G}_D)) * \ph = e_\lambda \conv ((\check{L} \conv \check{G}_D) * \ph) = e_\lambda \conv (\check{L} \conv (\check{G}_D * \ph)) .
}
Observe that also $\psi = \check{G}_D * \ph$ is a nonnegative function in $\test$, not identically equal to zero (because $\check{G}_D$ is compactly supported, nonnegative and not identically equal to zero). Hence,
\formula{
 (\alpha_{\lambda, D} - 1) e_\lambda * \ph & = e_\lambda \conv (\check{L} \conv \psi) = e_\lambda \conv (\check{L} * \psi) = (e_\lambda \conv \check{L}) * \psi = (e_\lambda * \psi) \conv \check{L} .
}
However, if $\beta_{\lambda, D} = e_\lambda * \ph(0)$ and $\gamma_{\lambda, D} = e_\lambda * \psi(0)$, then, by the argument already applied above, we have $e_\lambda * \ph = \beta_{\lambda, D} e_\lambda$ and $e_\lambda * \psi = \gamma_{\lambda, D} e_\lambda$. And since $\ph$ and $\psi$ are nonnegative and not identically equal to zero, we have $\beta_{\lambda, D} > 0$ and $\gamma_{\lambda, D} > 0$. We conclude that
\formula{
 (\alpha_{\lambda, D} - 1) \beta_{\lambda, D} e_\lambda & = \gamma_{\lambda, D} e_\lambda \conv \check{L} .
}
It follows that the sign of $\alpha_{\lambda, D} - 1$ is the same as the sign of $e_\lambda \conv \check{L}$, and our claim is proved. Additionally, $\lambda \in \Lambda_D$ if and only if $\alpha_{\lambda, D} = 1$, which is equivalent to $e_\lambda \conv \check{L} = 0$, that is, $\lambda \in \Lambda$. In other words, $\Lambda_D = \Lambda$.

\emph{Step 4.} We claim that there is no proper closed subgroup of $\G$ which contains the support of $\check{H}_D$ for every bounded open set $D$ such that $0 \in D$. We use the following interpretation of our assumption that $L$ is not concentrated on a proper closed subgroup of $\G$: if $X_t$ is the Lévy process generated by $L$, then the union of supports of all random variables $X_t - X_0$ is not contained in a proper closed subgroup of $\G$.

Suppose that $A$ is a compact set, $0 \notin A$, and $\check{H}_D(-A) = 0$ for every bounded open set $D$ such that $0 \in D$. Then
\formula{
 \pr^0(X_{\tau_D} \in A) & = 0
}
for every bounded open set $D$ such that $0 \in D$. By considering $D = \{x \in \G \setminus A : |x| < r\}$ and passing to the limit as $r \to \infty$, we find that with probability $\pr^0$ one, $X_t \notin A$ for all $t > 0$. Here we use the fact that with probability one, $\tau_D$ is equal to $\tau_{\G \setminus A}$ for $r$ large enough (and this, in turn, is a consequence of \emph{quasi-left continuity} of Lévy processes). In particular, $A$ is disjoint from the support of $X_t - X_0$ for every $t > 0$. It follows that the union of supports of measures $\check{H}_D$ (where $D$ is allowed to be an arbitrary bounded open set such that $0 \in D$) contains the union of supports of random variables $X_0 - X_t$ (where $t > 0$). The latter one generates a dense subgroup of $\G$, and hence the same is true for the former one. Our claim is proved.

\emph{Step 5.} Let $F_D$ denote the support of $\check{H}_D$. We choose a countable family of bounded open sets $D_n$ such that $0 \in D_n$, with the following property: the closure of the union of the supports $F_{D_n}$ is equal to the closure of the union of $F_D$ over all bounded open sets $D$ such that $0 \in D$. In order to do that, we may apply the following procedure: choose a countable base of open sets in $\G$, and for every basic open set $G$ which intersects the union of all $F_D$, choose a set $D_n$ so that $F_{D_n}$ intersects $G$.

Let $\check{H}_0$ be an arbitrary convex combination of the measures $\check{H}_{D_n}$ with positive coefficients. Then the support of $\check{H}_0$ is equal to the closure of the union of the supports $F_{D_n}$. But this set contains the union of all supports $F_D$, and by the result of the previous step, the latter is contained in no proper closed subgroup of $\G$. Hence, the support of $\check{H}_0$ is contained in no proper closed subgroup of $\G$. Since $h * \check{H}_{D_n} = h$ for every $n$, by Fubini's theorem we find that $h * \check{H}_0 = h$. Thus, we may apply Deny's theorem to conclude that $h$ is absolutely continuous with respect to the Lebesgue measure on $\G$, and the density function is given by
\formula{
 h(x) & = \int_{\Lambda_0} e^{\lambda x} \mu(d\lambda)
}
for a unique nonnegative measure $\mu$ on the set $\Lambda_0$ of those vectors $\lambda \in \R^d$ for which the function $e_\lambda(x) = e^{\lambda x}$ satisfies $e_\lambda * \check{H}_0 = e_\lambda$. It remains to show that $\Lambda_0 = \Lambda$.

Suppose that $\lambda \in \Lambda_0$. Then for every $n$ the convolution $e_\lambda * \check{H}_{D_n}$ is not everywhere infinite. By the result of step~3, $e_\lambda \conv \check{L}$ is well-defined, and the sign of $e_\lambda \conv \check{L}$ is the same as the sign of $e_\lambda * \check{H}_{D_n} - e_\lambda$, regardless of $n$. Applying again Fubini's theorem, we find that the sign of $e_\lambda \conv \check{L}$ is the same as the sign of $e_\lambda * \check{H}_0 - e_\lambda$, and since $\lambda \in \Lambda_0$, the latter is zero by assumption. Thus, $e_\lambda \conv \check{L} = 0$, that is, $\lambda \in \Lambda$. Conversely, if $\lambda \in \Lambda$, then $e_\lambda \conv \check{H}_D = e_\lambda$ for every bounded open set $D$ such that $0 \in D$, and thus $e_\lambda \conv \check{H}_0 = e_\lambda$, that is, $\lambda \in \Lambda_0$. This completes the proof.
\end{proof}

We remark that if $L$ is the one-dimensional Laplace operator and $D = (-r, r)$, then $\check{H}_D = \tfrac{1}{2} \delta_{-r} + \tfrac{1}{2} \delta_r$ has support contained in a proper closed subgroup $r \Z$ of $\G = \R$. However, we conjecture that for an arbitrary Lévy operator $L$ satisfying the assumption of Theorem~\ref{thm:positive:liouville} it is always possible find a single bounded open set $D$ such that the support of $\check{H}_D$ is not contained in a proper closed subgroup of~$\G$. For example, if $L$ is the one-dimensional Laplace operator and $D = (-a, b)$ with incommensurable $a, b > 0$, then $\check{H}_D = (a + b)^{-1} (b \delta_{-a} + a \delta_b)$ has support $\{-a, b\}$, which is not contained in a proper closed subgroup of $\G = \R$.

\begin{remark}
\label{rem:positive:liouville}
Let $L$ be an arbitrary Lévy operator on $\R^d$, and let $\G$ be the smallest closed subgroup of $\R^d$ such that $L$ is concentrated on $\G$. Then $L$ can be viewed as an operator which acts independently on each coset $x + \G$ of $\G$ in $\R^d$. Thus, $L$-harmonic functions or measures can be constructed independently on each coset $x + \G$ (as long as the resulting function is locally integrable on $\G$, or the resulting measure is locally finite on $\G$). On each coset, nonnegative $L$-harmonic measures are described by Theorem~\ref{thm:positive:liouville}.
\end{remark}

%
%                            ---------- o ----------
%

\section{An unusual $L$-harmonic function}
\label{sec:counter}

In this section we prove the following unexpected result.

\begin{theorem}[counterexample to the general Liouville's theorem]
\label{thm:counter}
There is a one-dimensional Lévy operator $L$, and a smooth function $h$, with the following properties:
\begin{enumerate}[label=(\alph*)]
\item $L$ is not concentrated on a proper subgroup of $\R$;
\item for every $\eps > 0$ we have $\lim_{|x| \to \infty} |x|^{-\eps} h(x) = 0$;
\item $L h(x) = 0$ for every $x \in \R$ (so that, in particular, the integral in the definition of $L h(x)$ is absolutely convergent);
\item $h$ is $L$-harmonic in the sense of tempered distributions: $\check{L} \schwconv h$ is well-defined and equal to zero, where $\check{L}$ is the convolution kernel of~$L$;
\end{enumerate}
but $h$ is not a polynomial.
\end{theorem}

More precisely, we consider a one-dimensional symmetric Lévy operator $L$ of the form
\formula{
 L f(x) & = f''(x) + \sum_{k = 0}^\infty p_k \bigl(f(x + x_k) + f(x - x_k) - 2 f(x)\bigr) ,
}
where $p_k = 2^{-k - 2}$ and $x_k$ is a rapidly increasing sequence.

Theorem~\ref{thm:counter} extends trivially to $\R^d$ by considering the Lévy operator which is the sum of operators $L$ defined above acting on each coordinate $x_j$, and the corresponding harmonic function which is the product of $h(x_j)$ for each coordinate $x_j$.

The construction of $h$ is somewhat technical. For this reason, we begin with a simpler, discrete variant of Theorem~\ref{thm:counter}.

%                            ---------- o ----------

\subsection{Discrete case}
\label{sec:counter:discrete}

In this section, we prove the following result, which will prepare us for the proof of Theorem~\ref{thm:counter}.

\begin{theorem}[counterexample to Liouville's theorem for lattice random walks]
\label{thm:counter:discrete}
Let
\formula{
 p_k & = 2^{-k - 2} & x_k & = 2^{2 k^2} .
}
There is a doubly infinite sequence $h(n)$ which satisfies
\formula{
 h(n) & = \sum_{k = 0}^\infty p_k \bigl(h(n + x_k) + h(n - x_k)\bigr)
}
for every $n \in \Z$, and such that for every $\eps > 0$,
\formula{
 \lim_{|n| \to \infty} \frac{h(n)}{|n|^\eps} = 0 ,
}
but $h(n)$ is not a polynomial sequence. Furthermore, for every $\eps > 0$, we have
\formula{
 \lim_{|n| \to \infty} \frac{1}{|n|^\eps} \sum_{k = 0}^\infty p_k (|h(n + x_k)| + |h(n - x_k)|) & = 0 .
}
\end{theorem}

The sequence $h(n)$ constructed in the proof is extremely sparse: we have $h(n) = 0$ unless $n = 0$ or $|n| = k + x_k$ for some $k = 0, 1, 2, \ldots\,$ More precisely, we set $h(0) = 1$, $h(n) = a_k$ if $|n| = k + x_k$, with appropriately chosen $a_k$, and $h(n) = 0$ for all other indices $n$.

Before we proceed with the construction of the sequence $h(n)$, we make the following observation.

\begin{lemma}
\label{lem:counter:sparse}
Suppose that $x_{k + 1} \ge 2 x_k$ and $x_k \ge 2 k$ for $k = 0, 1, 2, \ldots$ Then, for $n \in \Z$ and $k, m = 0, 1, 2, \ldots$\,,
\formula{
 |n \pm x_k| & = m + x_m && \text{implies that} & \text{$|n| = m$} && \text{or} && \text{$|n| \ge \tfrac{1}{2} x_m$} .
}
In particular, the above property holds when $x_k = 2^{2 k^2}$.
\end{lemma}

\begin{proof}
Suppose that $|n \pm x_k| = m + x_m$. If $k < m$, then
\formula{
 |n| & \ge m + x_m - x_k \ge x_m - x_{m - 1} \ge x_m - \tfrac{1}{2} x_m = \tfrac{1}{2} x_m .
}
If $k = m$, then either $|n| = m$ or $|n| = m + 2 x_m \ge \tfrac{1}{2} x_m$. Finally, if $k > m$, then
\formula{
 |n| & \ge x_k - m - x_m \ge x_{m + 1} - m - x_m \ge 2 x_m - m - x_m = x_m - m \ge \tfrac{1}{2} x_m . \qedhere
}
\end{proof}

From now on we let $p_k = 2^{-k - 2}$ and $x_k = 2^{2 k^2}$, as in Theorem~\ref{thm:counter:discrete}. For convenience, let us write
\formula*[eq:counter:operator]{
 L_0 f(n) & = \sum_{k = 0}^\infty p_k \bigl(f(n + x_k) + f(n - x_k) - 2 f(n)\bigr) \\
 & = \sum_{k = 0}^\infty p_k \bigl(f(n + x_k) + f(n - x_k)\bigr) - f(n)
}
whenever $f(n)$ is a doubly infinite sequence such that the above series converges absolutely (so that $L_0$ is a Lévy operator acting on $\Z$). The construction of the sequence $h(n)$ is an iterative procedure, which can be summarised as follows. In the initial step, we let $h_{-1}(n) = \ind_{\{0\}}(n)$. Next, for $m = 0, 1, 2, \ldots$ we define $h_m(n) = h_{m - 1}(n)$ except at two values of $n$, namely, $n = \pm (m + x_m)$. At these values we modify $h_m(n)$ in such a way that $L_0 h_m(m) = L_0 h_m(-m) = 0$. The key observation is that, by Lemma~\ref{lem:counter:sparse}, we also have $L_0 h_m(n) = L_0 h_{m - 1}(n) = 0$ if $|n| < m$. Finally, we define $h(n)$ to be the limit of $h_m(n)$ as $m \to \infty$.

We proceed with the detailed construction of $h(n)$. We let
\formula{
 h_{-1}(n) & = \ind_{\{0\}}(n) ,
}
and
\formula{
 h_m(n) & = h_{m - 1}(n) + a_m \ind_{\{-m - x_m, m + x_m\}}(n) \\
 & = \ind_{\{0\}}(n) + \sum_{j = 0}^m a_j \ind_{\{-j - x_j, j + x_j\}}(n)
}
for $m = 0, 1, 2, \ldots$\,, where
\formula{
 a_m & = -\frac{L_0 h_{m - 1}(m)}{p_m} = -\frac{1}{p_m} \biggl(\sum_{k = 0}^\infty p_k \bigl(h_{m - 1}(m + x_k) + h_{m - 1}(m - x_k)\bigr) - h_{m - 1}(m)\biggr)
}
if $m = 1, 2, \ldots$\,, and
\formula{
 a_0 & = -\frac{L_0 h_{-1}(0)}{2 p_0} = 2 .
}
Finally, we define
\formula{
 h(n) & = \lim_{m \to \infty} h_m(n) \\
 & = \ind_{\{0\}}(n) + \sum_{m = 0}^\infty a_m \ind_{\{-m - x_m, m + x_m\}}(n) .
}
Below we prove Theorem~\ref{thm:counter:discrete} by showing that the sequence $h(n)$ constructed above has all the desired properties. We break the proof into three lemmas. First, we prove that $h(n)$ is $L_0$-harmonic.

\begin{lemma}
\label{lem:counter:harm}
With the above definitions, we have
\formula{
 L_0 h(n) & = 0
}
for every $n \in \Z$.
\end{lemma}

\begin{proof}
By Lemma~\ref{lem:counter:sparse}, for a fixed $n \in \Z$ and all $k = 0, 1, 2, \ldots$ we have
\formula{
 |n + x_k| & \ne m + x_m \quad \text{and} \quad |n - x_k| \ne m + x_m && \text{if $m > |n|$.}
}
It follows that $h(n \pm x_k) = h_m(n \pm x_k)$, where $m = |n|$. Therefore,
\formula**{
 L_0 h(n) & = \sum_{k = 0}^\infty p_k \bigl(h(n + x_k) + h(n - x_k)\bigr) - h(n) \notag \\
 & = \sum_{k = 0}^\infty p_k \bigl(h_m(n + x_k) + h_m(n - x_k)\bigr) - h_m(n) \label{eq:counter:rec1} \\
 & = L_0 h_m(n) . \notag
}
We claim that $L_0 h_m(n) = 0$. If $n = m = 0$, then we simply have
\formula{
 L_0 h_0(0) & = p_0 \bigl(h_0(x_0) + h_0(-x_0)\bigr) - h_0(0) = 2 p_0 a_0 - 1 = 0 .
}
Suppose now that $n = m > 0$. By the definition of $h_m$, we have
\formula{
 L_0 h_m(m) & = L_0 h_{m - 1}(m) + a_m L_0 \ind_{\{-m - x_m, m + x_m\}}(m) .
}
Since $|m \pm x_k| \ne m + x_m$ if $k \ne m$, and also $|m - x_m| \ne m + x_m$, we find that
\formula{
 L_0 \ind_{\{-m - x_m, m + x_m\}}(m) & = p_m \ind_{\{-m - x_m, m + x_m\}}(m + x_m) - \ind_{\{-m - x_m, m + x_m\}}(m) = p_m .
}
Combining the above two identities and the definition $a_m = -p_m^{-1} L_0 h_{m - 1}(m)$ of $a_m$, we conclude that
\formula[eq:counter:rec2]{
 L_0 h_m(m) & = L_0 h_{m - 1}(m) + a_m p_m = 0 ,
}
as desired. By a similar argument (or by symmetry), we also have $L_0 h_m(-m) = 0$, and our claim is thus proved.
\end{proof}

In order to prove Theorem~\ref{thm:counter:discrete}, it remains to show appropriate estimates of $h(n)$. Recall that for $m = 1, 2, \ldots$,
\formula{
 p_m a_m & = -L_0 h_{m - 1}(m) = -\sum_{k = 0}^\infty p_k \bigl(h_{m - 1}(m + x_k+ h_{m - 1}(m - x_k)\bigr) + h_{m - 1}(m) .
}

\begin{lemma}
\label{lem:counter:aux}
With the above definitions, for every $\eps > 0$ there is $C_\eps$ such that
\formula[eq:counter:claim]{
 p_m |a_m| & \le \sum_{k = 0}^\infty p_k \bigl(|h_{m - 1}(m + x_k)| + |h_{m - 1}(m - x_k)|\bigr) + |h_{m - 1}(m)| \le C_\eps (1 + m)^\eps
}
for every $m = 0, 1, 2, \ldots$
\end{lemma}

\begin{proof}
Fix $\eps > 0$. Recall that $p_k = 2^{-k - 2}$ and $x_k = 2^{2 k^2}$. For $j$ large enough, say, $j > j_0$, we have
\formula{
 2^{j + 3} (1 + j)^\eps & \le (1 + \tfrac{1}{2} x_j)^\eps .
}
We choose $m_0$ so that
\formula{
 2^{j_0 + 3} (1 + j_0)^\eps & \le (1 + m_0)^\eps .
}
With this choice, we have the following property:
\formula[eq:counter:property]{
 2^{j + 3} (1 + j)^\eps & \le (1 + m)^\eps && \text{when $m \ge m_0$ and $m \ge \tfrac{1}{2} x_j$.}
}
We choose $C_\eps$ large enough, so that the desired bound~\eqref{eq:counter:claim} holds for $m = 1, 2, \ldots, m_0$, and we prove by induction that~\eqref{eq:counter:claim} also holds for $m > m_0$.

Suppose that for some $m > m_0$ formula~\eqref{eq:counter:claim} holds with $m$ replaced by any smaller number. The first inequality in~\eqref{eq:counter:claim} follows from the definition of $a_m$, so in order to complete the proof we only need to show the other inequality.

If $h_{m - 1}(m) \ne 0$, then $m = j + x_j$ and $h_{m - 1}(m) = a_j$ for some $j < m$. Since~\eqref{eq:counter:claim} holds with $m$ replaced by $j$, we obtain
\formula{
 |h_{m - 1}(m)| & = |a_j| \le \frac{C_\eps (1 + j)^\eps}{p_j} = C_\eps 2^{j + 2} (1 + j)^\eps .
}
Since $m > x_j$, we may apply~\eqref{eq:counter:property} to find that
\formula{
 |h_{m - 1}(m)| & \le \frac{C_\eps}{2} \, 2^{j + 3} (1 + j)^\eps \le \frac{C_\eps}{2} \, (1 + m)^\eps .
}
Similarly, if $h_{m - 1}(m \pm x_k) \ne 0$, then $|m \pm x_k| = j + x_j$ and $h_{m - 1}(m \pm x_k) = a_j$ for some $j < m$, so that again
\formula{
 |h_{m - 1}(m \pm x_k)| & = |a_j| \le \frac{C_\eps (1 + j)^\eps}{p_j} = C_\eps 2^{j + 2} (1 + j)^\eps .
}
By Lemma~\ref{lem:counter:sparse} we find that $m \ge \tfrac{1}{2} x_j$, and hence~\eqref{eq:counter:property} again leads to
\formula{
 |h_{m - 1}(m \pm x_k)| & \le \frac{C_\eps}{2} \, 2^{j + 3} (1 + j)^\eps \le \frac{C_\eps}{2} \, (1 + m)^\eps .
}
It follows that
\formula{
 & \sum_{k = 0}^\infty p_k \bigl(|h_{m - 1}(m + x_k)| + |h_{m - 1}(m - x_k)|\bigr) + |h_{m - 1}(m)| \\
 & \qquad \le \sum_{k = 0}^\infty 2 p_k \times \frac{C_\eps}{2} \, (1 + m)^\eps + \frac{C_\eps}{2} \, (1 + m)^\eps = C_\eps (1 + m)^\eps .
}
The proof is complete.
\end{proof}

\begin{lemma}
\label{lem:counter:bound}
With the above definitions, for every $\eps > 0$ we have
\formula{
 \lim_{|n| \to \infty} \frac{h(n)}{|n|^\eps} & = 0
}
and
\formula{
 \lim_{|n| \to \infty} \frac{1}{|n|^\eps} \sum_{k = 0}^\infty p_k \bigl(|h(n + x_k)| + |h(n - x_k)|\bigr) & = 0 .
}
\end{lemma}

\begin{proof}
If $n \ne 0$ and $|n| \ne m + x_m$ for every $m = 0, 1, 2, \ldots$\,, then $h(n) = 0$. If $|n| = m + x_m$ for some $m = 0, 1, 2, \ldots$\,, then $h(n) = a_m$, and, by Lemma~\ref{lem:counter:aux},
\formula{
 \frac{|h(n)|}{|n|^\eps} & = \frac{|a_m|}{(m + x_m)^\eps} \le \frac{C_\eps (1 + m)^\eps}{p_m (m + x_m)^\eps} = \frac{C_\eps 2^{m + 2} (1 + m)^\eps}{(m + 2^{2 m^2})^\eps} \, .
}
The right-hand side clearly converges to zero as $m \to \infty$, and the first part of the lemma follows.

To prove the other one, we consider $m = 0, 1, 2, \ldots$ and we recall that, as in~\eqref{eq:counter:rec1} and~\eqref{eq:counter:rec2}, we have
\formula{
 & \sum_{k = 0}^\infty p_k \bigl(|h(m + x_k)| + |h(m - x_k)|\bigr) \\
 & \qquad = \sum_{k = 0}^\infty p_k \bigl(|h_m(m + x_k)| + |h_m(m - x_k)|\bigr) \\
 & \qquad = p_m |a_m| + \sum_{k = 0}^\infty p_k (|h_{m - 1}(m + x_k)| + |h_{m - 1}(m - x_k)|\bigr) .
}
By Lemma~\ref{lem:counter:aux}, the right-hand side does not exceed $2 C_{\eps/2} (1 + m)^{\eps/2}$, and consequently
\formula{
 \lim_{m \to \infty} \frac{1}{m^\eps} \sum_{k = 0}^\infty p_k \bigl(|h(m + x_k)| + |h(m - x_k)|\bigr) & = 0 ,
}
as desired. In a similar way (or by symmetry),
\formula{
 \lim_{m \to \infty} \frac{1}{m^\eps} \sum_{k = 0}^\infty p_k \bigl(|h(-m + x_k)| + |h(-m - x_k)|\bigr) & = 0 ,
}
and the proof is complete.
\end{proof}

Theorem~\ref{thm:counter:discrete} is an immediate corollary of the above series of lemmas. We prove one additional property of the sequence $h(n)$, which in fact proves Theorem~\ref{thm:counter} without the assumption that $L$ is not concentrated on a proper closed subgroup of $\R$. The corresponding Lévy operator $L$ on $\R$ is given by the same expression as $L_0$:
\formula{
 L f(x) & = \sum_{k = 0}^\infty p_k \bigl(f(x + x_k) + f(x - x_k) - 2 f(x)\bigr) ,
}
but with $x \in \R$; see~\eqref{eq:counter:operator}.

\begin{lemma}
\label{lem:counter:conv}
With the above definitions, the convolution kernel $\check{L}$ of the Lévy operator $L$ is $\schw'$-convolvable with the measure
\formula{
 H & = \sum_{n = -\infty}^\infty h(n) \delta_n ,
}
and
\formula{
 \check{L} \schwconv H & = 0 .
}
\end{lemma}

\begin{proof}
By the last assertion of Theorem~\ref{thm:counter:discrete}, for a given $\eps > 0$, there is a constant $C_1$ such that
\formula{
 \sum_{k = 0}^\infty p_k \bigl(|h(n + x_k)| + |h(n - x_k)|\bigr) + |h(n)| & \le C_1 (1 + |n|)^\eps .
}
Furthermore, if $\ph, \psi \in \schw$ and $x \in \R$ is fixed, then there is a constant $C_2$ such that
\formula{
 |\ph| * |\psi|(x - n) & \le C_2 (1 + |n|)^{-2 - \eps} .
}
It follows that
\formula{
 \int_{\R} \sum_{n = -\infty}^\infty \biggl(\sum_{k = 0}^\infty p_k \bigl(|h(n + x_k)| + |h(n - x_k)|\bigr) + |\ph(n)|\biggr) |\ph(y - n)| |\psi(x - y)| dy & < \infty .
}
Thus, we may apply the result of Theorem~\ref{thm:counter:discrete} and Fubini's theorem to find that
\formula{
 0 & = \int_{\R} \sum_{n = -\infty}^\infty \biggl(\sum_{k = 0}^\infty p_k \bigl(h(n + x_k) + h(n - x_k)\bigr) - h(n)\biggr) \ph(y - n) \psi(x - y) dy \\
 & = \int_{\R} \sum_{n = -\infty}^\infty \sum_{k = 0}^\infty p_k h(n + x_k) \ph(y - n) \psi(x - y) dy \\
 & \qquad + \int_{\R} \sum_{n = -\infty}^\infty \sum_{k = 0}^\infty p_k h(n - x_k) \ph(y - n) \psi(x - y) dy \\
 & \qquad\qquad - \int_{\R} \sum_{n = -\infty}^\infty h(n) \ph(y - n) \psi(x - y) dy \displaybreak[0]\\
 & = \int_{\R} \sum_{n = -\infty}^\infty \sum_{k = 0}^\infty p_k h(n) \ph(y - n + x_k) \psi(x - y) dy \\
 & \qquad + \int_{\R} \sum_{n = -\infty}^\infty \sum_{k = 0}^\infty p_k h(n) \ph(y - n - x_k) \psi(x - y) dy \\
 & \qquad\qquad - \int_{\R} \sum_{n = -\infty}^\infty h(n) \ph(y - n) \psi(x - y) dy \displaybreak[0]\\
 & = \int_{\R} \sum_{n = -\infty}^\infty \sum_{k = 0}^\infty p_k h(n) \ph(y + x_k) \psi(x - y - n) dy \\
 & \qquad + \int_{\R} \sum_{n = -\infty}^\infty \sum_{k = 0}^\infty p_k h(n) \ph(y - x_k) \psi(x - y - n) dy \\
 & \qquad\qquad - \int_{\R} \sum_{n = -\infty}^\infty h(n) \ph(y) \psi(x - y - n) dy \displaybreak[0]\\
 & = \int_{\R} \biggl(\sum_{k = 0}^\infty p_k \bigl(\ph(y + x_k) + \ph(y + x_k)\bigr) - \ph(y)\biggr) \biggl(\sum_{n = -\infty}^\infty h(n) \psi(x - y - n)\biggr) dy \\
 & = \int_{\R} (\check{L} * \ph)(y) (H * \psi)(x - y) dy \\
 & = (\check{L} * \ph) * (H * \psi)(x) ,
}
as desired.
\end{proof}

We remark that the convolution of $H$ with an arbitrary Schwartz class function is a smooth $L$-harmonic function in the sense of tempered distributions for the Lévy operator $L$ (according to Definition~\ref{def:harmonic:schw}).

%                            ---------- o ----------

\subsection{Multiplication of distributions}
\label{sec:counter:product}

Before dealing with the case of Lévy operators on $\R$ (which are not concentrated on a proper subgroup of $\R$), we state the following counterintuitive result about $\schw'$-product of distributions.

\begin{corollary}
\label{cor:counter:prod}
There is a strictly positive continuous function $f$ on $\R$, and a nonzero tempered distribution $g$ on $\R$, such that $f \schwprod g = 0$.
\end{corollary}

\begin{proof}
Consider the one-dimensional Lévy operator $L$ defined in~\eqref{eq:counter:operator}, its convolution kernel $\check{L}$, and the corresponding characteristic exponent $\Psi$. Thus,
\formula{
 \Psi(\xi) & = 2 \sum_{k = 0}^\infty p_k \bigl(1 - \cos(x_k \xi)\bigr) = \sum_{k = 0}^\infty 2^{-k - 1} \bigl(1 - \cos(2^{2 k^2} \xi)\bigr)
}
is a Weierstrass-type nowhere differentiable function. Furthermore, let $h(n)$ be the doubly-infinite sequence from Theorem~\ref{thm:counter:discrete}, and let
\formula{
 H & = \sum_{n = -\infty}^\infty h(n) \delta_n .
}
By Lemma~\ref{lem:counter:conv}, $\check{L}$ and $H$ are $\schw'$-convolvable, and $\check{L} \schwconv H = 0$. The desired result essentially follows by the exchange formula: we have
\formula{
 \Psi \schwprod \fourier H & = -\fourier \check{L} \schwprod \fourier H = \fourier (\check{L} \schwconv H) = 0 .
}
Clearly, $\Psi$ is a nonnegative continuous function, and by~\eqref{eq:lk}, $\Psi$ is strictly positive everywhere except on $2 \pi \Z$ (recall that $\check{L}$ is concentrated on $\Z$, and it has an atom at $1$). On the other hand, $\fourier H$ is a periodic tempered distribution with period $2 \pi$, and since $h(n)$ is not a polynomial sequence, the support of $\fourier H$ is not contained in $2 \pi \Z$. Thus, to get the desired result, we only need to correct $\Psi$ and $\check{L}$ so that $\Psi$ is strictly positive everywhere.

One way to do this would be to replace $\check{L}$ by $\check{L} - \delta_0$ and repeat the construction of $h(n)$. However, there is a simpler solution: it is sufficient to define
\formula{
 f & = \Psi + \ph , \\
 g & = \psi \cdot \fourier H ,
}
where $\psi \in \schw$ is chosen in such a way that $\psi = 0$ on $2 \pi \Z$ and $\psi \cdot \fourier H$ is not identically zero, while $\ph \in \schw$ is a nonnegative function such that $\ph > 0$ on $2 \pi \Z$, but $\ph \cdot \psi = 0$ on $\R$. Indeed: $f$ is then a strictly positive continuous function, $g$ is a nonzero tempered distribution, and
\formula{
 f \schwprod g & = \Psi \schwprod (\psi \cdot \fourier H) + \ph \cdot (\psi \cdot \fourier H) \\
 & = \psi \cdot (\Psi \schwprod \fourier H) + (\ph \cdot \psi) \cdot \fourier H = 0 ,
}
as desired.
\end{proof}

%                            ---------- o ----------

\subsection{Continuous case: proof of Theorem~\ref{thm:counter}}
\label{sec:counter:continuous}

The proof of Theorem~\ref{thm:counter} is very similar to the argument used in the discrete case, in the proof of Theorem~\ref{thm:counter:discrete}. Thus, we omit some details and leave them to the interested reader.

We consider a Lévy operator similar to the one given in~\eqref{eq:counter:operator}, but with an additional one-dimensional Laplace operator. That is, we consider
\formula{
 L f(x) & = f''(x) + \sum_{k = 0}^\infty p_k \bigl(f(x + x_k) + f(x - x_k) - 2 f(x)\bigr) \\
 & = f''(x) + \sum_{k = 0}^\infty p_k \bigl(f(x + x_k) + f(x - x_k)\bigr) - f(x) ,
}
where again $p_k = 2^{-k - 2}$, and $x_k$ is a rapidly increasing sequence to be specified later. The construction of an $L$-harmonic function $h(x)$ is very similar to the construction of the sequence $h(n)$ in Section~\ref{sec:counter:discrete}, but for each $n \in \Z$ we replace the single number $h(n)$ by an appropriate compactly supported function $h(x)$, $x \in (n - \tfrac{1}{2}, n + \tfrac{1}{2})$. Additionally, we specify the value of $x_k$ on the fly, but in any case we will have
\formula[eq:counter:x]{
 x_0 & = 1, && x_{k + 1} \ge 2 x_k \text{ and } x_k \ge 4 k \text{ for } k = 0, 1, 2, \ldots \, ,
}
so that, in particular, Lemma~\ref{lem:counter:sparse} applies.

We consider a nonzero smooth even function $h_{-1}$ with support contained in $(-\tfrac{1}{2}, \tfrac{1}{2})$. Then, in step $m = 0, 1, 2, \ldots$\,, we define $h_m$ by appropriately modifying $h_{m - 1}$ on the intervals $|x| \in (m + x_m - \tfrac{1}{2}, m + x_m + \tfrac{1}{2})$ in such a way that $h_m$ is smooth, even, and $L h_m(x) = 0$ when $|x| \in (m - \tfrac{1}{2}, m + \tfrac{1}{2})$.

Let us describe more precisely step $m = 0, 1, 2, \ldots$ of the construction. Suppose that $h_{m - 1}$ and $x_0, x_1, \ldots, x_{m - 1}$ have already been defined. For $z \in (-\tfrac{1}{2}, \tfrac{1}{2})$ we define
\formula{
 h_m(x) & = h_{m - 1}(x) + \ph_m(x - m - x_m) + \ph_m(-x - m - x_m) \\
 & = h_{-1}(x) + \sum_{j = 0}^m \bigl(\ph_j(x - j - x_j) + \ph_j(-x - j - x_j)\bigr) ,
}
where $x_m$ is specified below,
\formula{
 \ph_m(z) & = -\frac{L h_{m - 1}(z + m)}{p_m} = -\frac{L h_{m - 1}(-z - m)}{p_m} \\
 & = h_{m - 1}''(z + m) + \sum_{k = 0}^{m - 1} p_k \bigl(h_{m - 1}(z + m + x_k) + h_{m - 1}(z + m - x_k)\bigr) - h_{m - 1}(z + m)
}
if $m = 1, 2, \ldots$\,, and
\formula{
 \ph_0(z) & = -\frac{L h_{-1}(z)}{2 p_0} = -2 h_{-1}''(z) + 2 h_{-1}(z)
}
if $m = 0$. For convenience, we set $\ph_m(z) = 0$ when $|z| \ge \tfrac{1}{2}$. Then $\ph_m$ is a smooth function with compact support in $(-\tfrac{1}{2}, \tfrac{1}{2})$. Note that in the above calculation of $L h_{m - 1}(z + m)$ we truncated the series at $k = m - 1$. This is because, as we now prove, all terms corresponding to $k \ge m$ are zero. Indeed: by construction, we have $h_{m - 1}(x) = 0$ when $|x| \ge (m - 1) + x_{m - 1} + \tfrac{1}{2}$, and if $k \ge m$, then, by~\eqref{eq:counter:x},
\formula{
 |m \pm x_k| & \ge x_k - m \ge x_m - m \ge \tfrac{1}{2} x_m + x_{m - 1} - m \ge m + x_{m - 1} .
}
Thus, $h_{m - 1}(z + m \pm x_k) = 0$, as claimed. In other words, the values of $x_k$ for $k \ge m$ are not needed in order to evaluate $\ph_m(z) = -p_m^{-1} L h_{m - 1}(z + m)$, as long as condition~\eqref{eq:counter:x} is satisfied, and this allows us to specify the value of $x_m$ only in step $m$. If $m > 0$, then we choose $x_m$ to be an integer large enough, so that $x_m \ge 2 x_{m - 1}$, $x_m \ge 4 m$, $x_m \ge 2^{2 m^2}$, and
\formula[eq:counter:xk]{
 \log x_m & \ge \sup \{|\ph_m(z)| + |\ph_m''(z)| : z \in (-\tfrac{1}{2}, \tfrac{1}{2})\} .
}
If $m = 0$, we simply let $x_0 = 1$.

We now define
\formula{
 h(x) & = \lim_{m \to \infty} h_m(x) \\
 & = h_{-1}(x) + \sum_{m = 0}^\infty \bigl(\ph_m(x - m - x_m) + \ph_m(-x - m - x_m)\bigr) .
}
We stress that $h$ is a smooth function with a very sparse support; namely, we have
\formula{
 h(x) & = \begin{cases}
  \ph_0(x) & \text{if $x \in (-\tfrac{1}{2}, \tfrac{1}{2})$,} \\
  \ph_m(x - m - x_m) & \text{if $x \in (m + x_m - \tfrac{1}{2}, m + x_m + \tfrac{1}{2})$ with $m = 0, 1, \ldots$\,,} \\
  \ph_m(-x - m - x_m) & \text{if $x \in (-m - x_m - \tfrac{1}{2}, -m - x_m + \tfrac{1}{2})$ with $m = 0, 1, \ldots$\,,} \\
  0 & \text{otherwise} .
 \end{cases}
}
We now follow closely the arguments used in the discrete case in Sections~\ref{sec:counter:discrete} and~\ref{sec:counter:product}.

\begin{lemma}
\label{lem:counter:harm:smooth}
With the above definitions, we have
\formula{
 L h(x) & = 0
}
for every $x \in \R$.
\end{lemma}

\begin{proof}
The argument is almost exactly the same as in the proof of Lemma~\ref{lem:counter:harm}, except that we replace $n$ by $n + z$, with $z \in (-\frac{1}{2}, \frac{1}{2})$, and we need to consider $z = \tfrac{1}{2}$ separately.

By construction, $h(n + \tfrac{1}{2}) = h''(n + \tfrac{1}{2}) = 0$ for every $n \in \Z$. Thus, $L h(n + \tfrac{1}{2}) = 0$, and we only need to show that $L h(z + n) = 0$ when $n \in \Z$ and $z \in (-\tfrac{1}{2}, \tfrac{1}{2})$. By Lemma~\ref{lem:counter:sparse}, for every $k = 0, 1, 2, \ldots$ we have
\formula{
 |n \pm x_k| & \ne m + x_m && \text{if $m > |n|$,}
}
and thus $h(z + n \pm x_k) = h_m(z + n \pm x_k)$, where $m = |n|$. Therefore,
\formula{
 L h(z + n) & = h''(z + n) + \sum_{k = 0}^\infty p_k \bigl(h(z + n + x_k) + h(z + n - x_k)\bigr) - h(z + n) \\
 & = h_m''(z + n) + \sum_{k = 0}^\infty p_k \bigl(h_m(z + n + x_k) + h_m(z + n - x_k)\bigr) - h_m(z + n) \\
 & = L h_m(z + n) .
}
We claim that $L h_m(z + n) = 0$. If $n = m = 0$, then
\formula{
 L h_0(z) & = h_0''(0) + p_0 \bigl(h_0(z + x_0) + h_0(z - x_0)\bigr) - h_0(z) \\
 & = h_{-1}''(z) + 2 p_0 \ph_0(z) - h_{-1}(z) = 0
}
by the definitions of $h_0$ and $\ph_0$. Suppose now that $n = m > 0$. By the definition of $h_m$, we have
\formula{
 L h_m(z + m) & = L h_{m - 1}(z + m) + L \ph_m(z - x_m) + L \ph_m(z + 2 m + x_m) .
}
Note that $|z - x_m| > \tfrac{1}{2}$, $|z - x_m \pm x_k| > \tfrac{1}{2}$ if $k \ne m$, and also $|z - 2 x_m| > \tfrac{1}{2}$. Thus,
\formula{
 L \ph_m(z - x_m) & = \ph_m''(z - x_m) - \ph_m(z - x_m) \\
 & \qquad + \sum_{k = 0}^\infty p_k \bigl(\ph_m(z - x_m + x_k) + \ph_m(z - x_m - x_k)\bigr) \\
 & = p_m \ph_m(z) = -L h_{m - 1}(m + z) .
}
Similarly, $|z + 2 m + x_m| > \tfrac{1}{2}$ and $|z + 2 m + x_m \pm x_k| > \tfrac{1}{2}$ for every $k$, and hence
\formula{
 L \ph_m(z + 2 m + x_m) & = \ph_m''(z + 2 m + x_m) - \ph_m(z + 2 m + x_m) \\
 & \qquad + \sum_{k = 0}^\infty p_k \bigl(\ph_m(z + 2 m + x_m + x_k) + \ph_m(z + 2 m + x_m - x_k)\bigr) = 0 .
}
Therefore,
\formula{
 L h_m(z + m) & = L h_{m - 1}(z + m) - L h_{m - 1}(z + m) = 0 .
}
By a similar argument (or by symmetry), we also have $L h_m(-z - m) = 0$, and our claim is thus proved.
\end{proof}

\begin{lemma}
\label{lem:counter:aux:smooth}
If $x_k$ grows sufficiently fast, then, with the above definitions, for every $\eps > 0$ there is a constant $C_\eps$ such that
\formula*[eq:counter:claim:smooth]{
 p_m |\ph_m(z)| & \le |h_{m - 1}''(z + m)| + \sum_{k = 0}^\infty p_k \bigl(|h_{m - 1}(z + m + x_k)| + |h_{m - 1}(x - x_k)|\bigr) \\
 & \qquad + |h_{m - 1}(z + m)| \le C_\eps (1 + m)^\eps
}
for $m = 0, 1, 2, \ldots$ and $z \in (-\tfrac{1}{2}, \tfrac{1}{2})$.
\end{lemma}

\begin{proof}
The proof is actually simpler than the proof of Lemma~\ref{lem:counter:aux}, due to flexibility in the choice of $x_m$. The first inequality in~\eqref{eq:counter:claim:smooth} follows by the definition of $\ph_m$, and so we are left with the proof of the other inequality.

Let $\eps > 0$. By~\eqref{eq:counter:xk}, there is a constant $C_\eps$ such that
\formula[eq:counter:aux:xk]{
 \sup \{|\ph_m(z)| + |\ph_m''(z)| : z \in (-\tfrac{1}{2}, \tfrac{1}{2})\} & \le C_\eps (1 + x_m)^\eps
}
for $m = 0, 1, 2, \ldots$\,

Fix $m = 0, 1, 2, \ldots$\, If $h_{m - 1}(z + m) \ne 0$ for some $z \in (-\tfrac{1}{2}, \tfrac{1}{2})$, then $m = j + x_j$ and $h_{m - 1}(z + m) = \ph_j(z)$ for some $j < m$. By~\eqref{eq:counter:aux:xk}, we obtain
\formula{
 |h_{m - 1}(z + m)| + |h_{m - 1}''(z + m)| & = |\ph_j(z)| + |\ph_j''(z)| \le C_\eps (1 + x_j)^\eps \le C_\eps (1 + m)^\eps .
}

Similarly, if $h_{m - 1}(z + m \pm x_k) \ne 0$, then $|m \pm x_k| = j + x_j$ and $h_{m - 1}(z + m \pm x_k) = \ph_j(\pm z)$ for some $j < m$, so that again
\formula{
 |h_{m - 1}(z + m \pm x_k)| & = |\ph_j(\pm z)| \le C_\eps (1 + x_j)^\eps .
}
By Lemma~\ref{lem:counter:sparse} we find that $m \ge \tfrac{1}{2} x_j$, and hence
\formula{
 |h_{m - 1}(z + m \pm x_k)| & \le C_\eps (1 + 2 m)^\eps \le 2^\eps C_\eps (1 + m)^\eps .
}
It follows that
\formula{
 & |h_{m - 1}''(z + m)| + \sum_{k = 0}^\infty p_k \bigl(|h_{m - 1}(z + m + x_k)| + |h_{m - 1}(z + m - x_k)|\bigr) + |h_{m - 1}(z + m)| \\
 & \qquad \le C_\eps (1 + m)^\eps + 2^\eps C_\eps (1 + m)^\eps \le 2^{1 + \eps} C_\eps (1 + m)^\eps .
}
The proof is complete.
\end{proof}

\begin{lemma}
\label{lem:counter:bound:smooth}
With the above definitions, for every $\eps > 0$ we have
\formula{
 \lim_{|x| \to \infty} \frac{|h(x)| + |h''(x)|}{|x|^\eps} & = 0 ,
}
and
\formula{
 \lim_{|x| \to \infty} \frac{1}{|x|^\eps} \sum_{k = 0}^\infty p_k \bigl(|h(x + x_k)| + |h(x - x_k)|\bigr) & = 0 .
}
\end{lemma}

\begin{proof}
The proof is very similar to the proof of Lemma~\ref{lem:counter:bound}. By definition, if $n \in \Z$, $z \in (-\tfrac{1}{2}, \tfrac{1}{2})$, $|n| \ge 2$ and $h(z + n) \ne 0$ or $h''(z + n) \ne 0$, then $|n| = m + x_m$ for some $m = 1, 2, \ldots$\,, and thus, by~\eqref{eq:counter:xk},
\formula{
 \frac{|h(z + n)| + |h''(z + n)|}{|z + n|^\eps} & = \frac{|\ph_m(z)| + |\ph_m''(z)|}{|z + n|^\eps} \le \frac{\log x_m}{|n|^\eps} \le \frac{\log x_m}{x_m^\eps} \, .
}
The right-hand side clearly converges to zero as $m \to \infty$, and the first part of the lemma follows.

To prove the other one, we consider $m = 0, 1, 2, \ldots$ and $z \in (-\tfrac{1}{2}, \tfrac{1}{2})$, and we recall that, as in the proof of Lemma~\ref{lem:counter:harm:smooth}, we have
\formula{
 & \sum_{k = 0}^\infty p_k \bigl(|h(z + m + x_k)| + |h(z + m - x_k)|\bigr) \\
 & \qquad = \sum_{k = 0}^\infty p_k \bigl(|h_m(z + m + x_k)| + |h_m(z + m - x_k)|\bigr) \\
 & \qquad = p_m |\ph_m(z)| + \sum_{k = 0}^\infty p_k (|h_{m - 1}(z + m + x_k)| + |h_{m - 1}(z + m - x_k)|\bigr) .
}
By Lemma~\ref{lem:counter:aux:smooth}, the right-hand side does not exceed $C_{\eps/2} (1 + m)^{\eps/2}$, and thus
\formula{
 \lim_{m \to \infty} \frac{1}{m^\eps} \sum_{k = 0}^\infty p_k \bigl(|h(z + m + x_k)| + |h(z + m - x_k)|\bigr) & = 0 ,
}
uniformly with respect to $z \in (-\tfrac{1}{2}, \tfrac{1}{2})$. By a similar argument (or by symmetry),
\formula{
 \lim_{m \to \infty} \frac{1}{m^\eps} \sum_{k = 0}^\infty p_k \bigl(|h(-z - m + x_k)| + |h(-z - m - x_k)|\bigr) & = 0
}
uniformly with respect to $z \in (-\tfrac{1}{2}, \tfrac{1}{2})$, and the proof is complete.
\end{proof}

\begin{lemma}
\label{lem:counter:conv:smooth}
With the above definitions, the function $h$ corresponds to a tempered distribution, which is $\schw'$-convolvable with the convolution kernel $\check{L}$ of the Lévy operator $L$, and we have $\check{L} \schwconv h = 0$.
\end{lemma}

\begin{proof}
The argument is virtually the same as in the proof of Lemma~\ref{lem:counter:conv}. By Lemma~\ref{lem:counter:bound:smooth}, for a given $\eps > 0$, there is a constant $C_1$ such that
\formula{
 |h''(z)| + \sum_{k = 0}^\infty p_k \bigl(|h(z + x_k)| + |h(z - x_k)|\bigr) + |h(z)| & \le C_1 (1 + |z|)^\eps .
}
Furthermore, if $\ph, \psi \in \schw$ and $x \in \R$ is fixed, then there is a constant $C_2$ such that
\formula{
 |\ph| * |\psi|(x - z) & \le C_2 (1 + |z|)^{-2 - \eps} .
}
It follows that
\formula{
 \int_{\R} \int_{\R} \biggl(|h''(z)| + \sum_{k = 0}^\infty p_k \bigl(|h(z + x_k)| + |h(z - x_k)|\bigr) - h(z)\biggr) |\ph(y - z)| |\psi(x - y)| dz dy & < \infty .
}
Thus, we may apply Lemma~\ref{lem:counter:harm:smooth}, Fubini's theorem and integration by parts to find that
\formula{
 0 & = \int_{\R} \int_{\R} \biggl(h''(z) + \sum_{k = 0}^\infty p_k \bigl(h(z + x_k) + h(z - x_k)\bigr) - h(z)\biggr) \ph(y - z) \psi(x - y) dz dy \\
 & = \int_{\R} \int_{\R} h''(z) \ph(y - z) \psi(x - y) dz dy \displaybreak[0]\\
 & \qquad + \int_{\R} \int_{\R} \sum_{k = 0}^\infty p_k h(z + x_k) \ph(y - z) \psi(x - y) dz dy \\
 & \qquad\qquad + \int_{\R} \int_{\R} \sum_{k = 0}^\infty p_k h(z - x_k) \ph(y - z) \psi(x - y) dz dy \\
 & \qquad\qquad\qquad - \int_{\R} \int_{\R} h(z) \ph(y - z) \psi(x - y) dz dy \displaybreak[0]\\
 & = \int_{\R} \int_{\R} h(z) \ph''(y - z) \psi(x - y) dz dy \displaybreak[0]\\
 & \qquad + \int_{\R} \int_{\R} \sum_{k = 0}^\infty p_k h(z) \ph(y - z + x_k) \psi(x - y) dz dy \\
 & \qquad\qquad + \int_{\R} \int_{\R} \sum_{k = 0}^\infty p_k h(z) \ph(y - z - x_k) \psi(x - y) dz dy \\
 & \qquad\qquad\qquad - \int_{\R} \int_{\R} h(z) \ph(y - z) \psi(x - y) dz dy \displaybreak[0]\\
 & = \int_{\R} \int_{\R} h(z) \ph''(y) \psi(x - y - z) dz dy \displaybreak[0]\\
 & \qquad + \int_{\R} \int_{\R} \sum_{k = 0}^\infty p_k h(z) \ph(y + x_k) \psi(x - y - z) dz dy \\
 & \qquad\qquad + \int_{\R} \int_{\R} \sum_{k = 0}^\infty p_k h(z) \ph(y - x_k) \psi(x - y - z) dz dy \\
 & \qquad\qquad\qquad - \int_{\R} \int_{\R} h(z) \ph(y) \psi(x - y - z) dz dy \displaybreak[0]\\
 & = \int_{\R} \biggl(\ph''(y) + \sum_{k = 0}^\infty p_k \bigl(\ph(y + x_k) + \ph(y + x_k)\bigr) - \ph(y)\biggr) \biggl(\int_{\R} h(z) \psi(x - y - z) dz\biggr) dy \\
 & = \int_{\R} (\check{L} * \ph)(y) (h * \psi)(x - y) dy \\
 & = (\check{L} * \ph) * (h * \psi)(x) ,
}
and the proof is complete.
\end{proof}

Theorem~\ref{thm:counter} follows directly from the above series of lemmas.

%
%                            ---------- o ----------
%

\section{Signed harmonic functions}
\label{sec:signed}

In this final section of the article we prove various variants of Liouville's theorem for signed polynomially bounded functions using Fourier transform approach. We begin with an abstract result in Theorem~\ref{thm:liouville}, and then, by choosing an appropriate Wiener-type algebra $W$, we obtain specific Liouville's theorems as corollaries.

%                            ---------- o ----------

\subsection{General result}
\label{sec:signed:general}

Recall that according to Definition~\ref{def:algebra}, an algebra $W$ of continuous functions on $\R^d$ is a \emph{Wiener-type algebra} if every $\Phi \in W$ corresponds to a tempered distribution, $\ph \cdot \Phi \in W$ whenever $\ph \in \schw$ and $\Phi \in W$, and the following variant of Wiener's $1/f$ theorem holds:
\formula{
 \begin{array}{c}\text{if $K \subseteq \R^d$ is a compact set, $\Phi \in W$ and $\Phi(\xi) \ne 0$ for every $\xi \in K$,} \\ \text{then there is $\tilde{\Phi} \in W$ such that $\Phi(\xi) \tilde{\Phi}(\xi) = 1$ for every $\xi \in K$.}\end{array}
}
A tempered distribution $\Psi$ is said to belong to $W$ locally on an open set $U$ if for every compact set $K \subseteq U$ there is a tempered distribution $\Phi \in W$ such that $\Psi = \Phi$ in a neighbourhood of $K$. In particular, in this case the restriction of $\Psi$ to $U$ is given by a continuous function, and if $\ph \in \schw$ has a compact support contained in $U$, then $\ph \cdot \Psi$ is an element of $W$.

Using the notation introduced in Section~\ref{sec:pre:distr}, Definition~\ref{def:act} reads as follows. A tempered distribution $H$ is said to \emph{act} on $W$ if for every $\Phi, \Psi \in W$ we have:
\formula{
 (H \schwprod \Phi) \schwprod \Psi & = H \schwprod (\Phi \cdot \Psi) .
}
The main result of this section is the following extension of Theorem~\ref{thm:liouville:0}. Note that we do not require $\check{L}$ to be the convolution kernel of a Lévy operator.

\begin{theorem}[Liouville's theorem factory]
\label{thm:liouville}
Let $W$ be a Wiener-type algebra of continuous functions on $\R^d$, and let $\check{L}$ be a tempered distribution. Suppose that $\fourier \check{L}$ belongs to $W$ locally on an open set $U$, and assume that $\fourier \check{L}(\xi) \ne 0$ for every $\xi \in U$. Let $h$ be a tempered distribution such that $\fourier h$ acts on $W$, and such that $\check{L} \schwconv h = 0$. Then the spectrum of $h$ is contained in $\R^d \setminus U$. In other words: the restriction of $\fourier h$ to $U$ is zero.

In particular, if $U = \R^d \setminus \{0\}$, then $h$ is a polynomial.
\end{theorem}

\begin{proof}
Fix a compact subset $K$ of $U$. Suppose that the spectrum of $\ph \in \schw$ is a compact subset of $U$ and $\fourier \ph(\xi) \ne 0$ for every $\xi \in K$. Define $f = -\check{L} * \ph$, $\Psi = -\fourier \check{L}$, $\Phi = \fourier f$ and $H = \fourier h$. Observe that
\formula{
 \Phi & = \fourier f = -\fourier(\check{L} * \ph) = -\fourier \ph \cdot \fourier \check{L} = \fourier \ph \cdot \Psi .
}
By assumption, $\fourier \ph \in \schw$, the support of $\fourier \ph$ is a compact subset of $U$, and $\Psi$ belongs to $W$ locally on $U$. Thus, $\Phi \in W$. Furthermore,
\formula{
 0 = (\check{L} \schwconv h) * \ph & = h \schwconv (\check{L} * \ph) = -h \schwconv f ,
}
and hence, by the Fourier exchange formula,
\formula{
 0 & = \fourier(h \schwconv f) = H \schwprod \Phi .
}
On the other hand, $\Phi(\xi) = \fourier \ph(\xi) \Psi(\xi) \ne 0$ for every $\xi \in K$, and therefore there is $\tilde{\Phi} \in W$ such that $\Phi(\xi) \tilde{\Phi}(\xi) = 1$ for every $\xi \in K$. Since $H$ acts on $W$, we conclude that
\formula{
 0 & = (H \schwprod \Phi) \schwprod \tilde{\Phi} = H \schwprod (\Phi \cdot \tilde{\Phi}) .
}
Recall that $\Phi(\xi) \tilde{\Phi}(\xi) = 1$ for every $\xi \in K$. By definition, multiplication of distributions is a local operation. Hence, in the interior of $K$, we have
\formula{
 H & = H \schwprod 1 = H \schwprod (\Phi \cdot \tilde{\Phi}) = 0 .
}
Since $K$ is an arbitrary subset of $U$, we conclude that $H = 0$ on $U$, as desired.

If $U = \R^d \setminus \{0\}$, then $H = \fourier h$ is supported in $\R^d \setminus U = \{0\}$, and hence $h$ is necessarily a polynomial (see Sections~2.6 and~6.3.2 in~\cite{vladimirov}).
\end{proof}

%                            ---------- o ----------

\subsection{Operators with smooth symbols}
\label{sec:signed:smooth}

It is straightforward to verify that $W = \schw$, the Schwartz class of rapidly decaying functions, is a Wiener-type algebra. Clearly, $\Psi$ belongs to $W$ locally on $U$ if and only if $\Psi$ is smooth on $U$. Finally, if $\Phi, \Psi \in W$ and $H$ is an arbitrary tempered distribution, we have $(\Phi \cdot \Psi) \cdot H = \Phi \cdot (\Psi \cdot H)$, and so every tempered distribution acts on $W$. This leads to the following statement, which is a minor extension of Theorem~3.2 in~\cite{bs}.

\begin{corollary}[Liouville's theorem for operators with smooth symbols]
\label{cor:liouville:smooth}
Let $\check{L}$ be a tempered distribution. Suppose that on an open set $U$, $\fourier \check{L}$ corresponds to a smooth function without zeroes. Let $h$ be a tempered distribution such that $\check{L} \schwconv h = 0$. Then the spectrum of $h$ is contained in $\R^d \setminus U$.

In particular, if $U = \R^d \setminus \{0\}$, then $h$ is a polynomial.
\end{corollary}

\begin{example}
Let $\alpha \in (0, 2)$ and let $L = -(-\Delta)^{\alpha/2}$ be the fractional Laplace operator. If $\check{L}$ denotes the corresponding convolution kernel, then $\fourier \check{L}(\xi) = -|\xi|^\alpha$. Since $\fourier \check{L}$ is smooth in $\R^d \setminus \{0\}$, we find that the only polynomially bounded $L$-harmonic functions (in the sense of tempered distributions) are polynomials.

However, $L h$ is not well-defined if $h$ is a polynomial of degree $\lceil \alpha \rceil$ or higher, and so all $L$-harmonic functions are constant when $\alpha \in (0, 1]$, and all $L$-harmonic functions are affine when $\alpha \in (1, 2)$.

This is exactly the main result of~\cite{fall} (Theorem~1.1 therein) and of~\cite{cdl} (Theorem~1.3 therein).
\end{example}

\begin{remark}
\label{rem:lizorkin}
Suppose that $L = -(-\Delta)^{\alpha/2}$ is the fractional Laplace operator, or, more generally, that $\fourier \check{L}$ is a smooth function on $\R^d \setminus \{0\}$. As it was kindly pointed out by the anonymous referee, in this case $L f$ can be defined \emph{up to a polynomial} for an arbitrary tempered distribution $f$ in the following way.

Let $\schw_0$ denote the space of Schwartz class functions which are orthogonal in $L^2(\R^d)$ to all polynomials, the so-called \emph{Lizorkin space} (often denoted by $\Phi$ in literature). The corresponding dual space $\schw_0'$ is the class of tempered distributions, where any two distributions differing by a polynomial are identified. Then $L$ is a well-defined linear operator on $\schw_0'$. If $h \in \schw'$ is a tempered distribution, then we say that $h$ is $L$-harmonic up to a polynomial if the equivalence class $h_0 \in \schw_0'$ of $h$ satisfies $L h_0 = 0$ in $\schw_0'$. Clearly, this condition is equivalent to Definition~\ref{def:harmonic:schw} with $\schw$ replaced by $\schw_0$.

We claim that our Theorem~\ref{thm:liouville} can be adapted to include the above setting. Suppose that a tempered distribution $h \in \schw'$ is $L$-harmonic up to a polynomial. Choose an arbitrary $\psi \in \schw$ such that $\fourier \psi$ is equal to zero in a neighbourhood of $0$. Then $h * \psi$ is $L$-harmonic in the sense of tempered distributions, and so, by Corollary~\ref{cor:liouville:smooth}, $h * \psi$ is a polynomial. Equivalently, $\fourier (h * \psi) = \fourier h \cdot \fourier \psi$ is a tempered distribution supported in $\{0\}$. Since $\fourier \psi$ is an arbitrary Schwartz class function vanishing in a neighbourhood of $0$, we conclude that $\fourier h$ is a tempered distribution supported in $\{0\}$, and hence $h$ is a polynomial.

The Lizorkin space was introduced in 1960s in the study of spaces of fractional smoothness, and it became a standard tool in the theory of hypersingular integrals; we refer to \emph{Bibliographical notes to Chapter~2} in~\cite{samko} for a detailed discussion. The above variant of Liouville's theorem for the fractional Laplace operator is due to Dipierro, Savin and Valdinoci; their result in~\cite{dsv} (Theorem~1.5 therein) uses a different, more PDE-oriented definition of functions $L$-harmonic up to a polynomial, but it can be shown to be equivalent to the one given here.
\end{remark}

%                            ---------- o ----------

\subsection{Bounded $L$-harmonic functions}
\label{sec:signed:bounded}

Let $W$ be the Wiener algebra: the class of Fourier transforms of integrable functions. It is straightforward to see that $W$ satisfies the first two conditions of Definition~\ref{def:algebra}, and the last one is a variant of the classical Wiener's $1/f$ theorem. We refer to the Division Lemma in Section~3 in~\cite{cs} for the proof in dimension one, and to Section~\ref{sec:signed:oper} below for a more general statement.

If $\check{L}$ is an integrable distribution, $r > 0$ and $\ph \in \schw$ satisfies $\fourier \ph(\xi) = 1$ when $|\xi| < r$, then $\check{L} * \ph$ is an integrable function with Fourier transform $\fourier \ph \cdot \fourier \check{L}$ in $W$. Since $r$ is arbitrary, we see that $\fourier \check{L}$ belongs to $W$ locally on $\R^d$. In particular, Fourier symbols of Lévy operators belong to $W$ locally on $\R^d$.

Observe that if $h$ is a bounded distribution and $\Phi, \Psi \in W$, then there are integrable functions $f, g$ such that $\Phi = \fourier f$ and $\Psi = \fourier g$. Since $h \schwconv (f * g) = h \schwconv (f \schwconv g) = (h \schwconv f) \schwconv g$, by the exchange formula we find that $\fourier h \schwprod (\Phi \cdot \Psi) = (\fourier h \schwprod \Phi) \schwprod \Psi$. Thus, Fourier transforms of bounded distributions act on $W$.

The above observations immediately lead to the following minor extension of the general Liouville's theorem for bounded $L$-harmonic functions given in Theorem~1.1 in~\cite{adej} and in Theorem~4.4 in~\cite{bs}.

\begin{corollary}[Liouville's theorem for bounded functions]
\label{cor:liouville:bounded}
Let $\check{L}$ be an integrable distribution. Suppose that $\fourier \check{L}$, which is necessarily a continuous function, has no zeroes in an open set $U$. Let $h$ be a bounded distribution such that $\check{L} \schwconv h = 0$. Then the spectrum of $h$ is contained in $\R^d \setminus U$.

In particular, if $U = \R^d \setminus \{0\}$, then $h$ is a constant.
\end{corollary}

%                            ---------- o ----------

\subsection{Operators with finite generalised moments}
\label{sec:signed:oper}

We define the following class of Wiener-type algebras, parameterised by auxiliary functions $\profile$. Examples of admissible functions $\profile$ are discussed at the end of this section. Here we remark that a typical choice is $\profile(x) \approx |x|^\alpha$, where $\alpha > 0$ is a parameter, and that by considering $\profile(x) = 1$ we recover the particular case of bounded solutions, discussed in Section~\ref{sec:signed:bounded}.

\begin{definition}
\label{def:oper}
Suppose that $\profile$ is a nonnegative function on $\R^d$ with the following properties:
\begin{enumerate}[label=(\alph*)]
\item\label{it:oper:1} we have $1 + \profile(x + y) \le (1 + \profile(x)) (1 + \profile(y))$ for every $x, y \in \R^d$;
\item\label{it:oper:2} for some positive constants $c, q$ we have $\profile(x) \le c (1 + |x|)^q$.
\end{enumerate}
Define $W_\profile^\infty$ to be the class of Fourier transforms of integrable functions $f$ such that also $\profile f$ is integrable.
\end{definition}

Our main goal in this section is to prove that $W_\profile^\infty$ is a Wiener-type algebra according to Definition~\ref{def:algebra}.

\begin{lemma}
\label{lem:oper:1}
The set $W_\profile^\infty$ defined in Definition~\ref{def:oper} is an algebra of continuous functions which satisfies conditions~\ref{it:al:1} and~\ref{it:al:2} of Definition~\ref{def:algebra}.
\end{lemma}

\begin{proof}
Clearly, $W_\profile^\infty$ is a linear space. If $\Phi \in W_\profile^\infty$, then $\Phi$ is the Fourier transform of an integrable function $f$, and hence $\Phi$ is a bounded continuous function. Thus, $W_\profile^\infty$ is a class of continuous functions, and every $\Phi \in W_\profile^\infty$ corresponds to a tempered distribution.

We claim that $W_\profile^\infty$ is indeed an algebra of functions. If $\Phi, \Psi \in W_\profile^\infty$, then $\Phi = \fourier f$ and $\Psi = \fourier g$ for some integrable functions $f, g$ such that also $\profile f, \profile g$ are integrable. It follows that $h = f * g$ is an integrable function, and by condition~\ref{it:oper:1} in Definition~\ref{def:oper},
\formula*[eq:oper:conv]{
 \int_{\R^d} |(1 + \profile(x)) h(x)| dx & = \int_{\R^d} \biggl| (1 + \profile(x)) \int_{\R^d} f(y) g(x - y) dy \biggr| dx \\
 & \le \int_{\R^d} \int_{\R^d} (1 + \profile(x)) |f(y)| g(x - y)| dy dx \\
 & = \int_{\R^d} \int_{\R^d} (1 + \profile(y + z)) |f(y)| |g(z)| dy dz \\
 & \le \int_{\R^d} (1 + \profile(y)) |f(y)| dy \cdot \int_{\R^d} (1 + \profile(z)) |g(z)| dz < \infty ,
}
that is, also $\profile h$ is integrable. Since $\Phi \cdot \Psi = \fourier f \cdot \fourier g = \fourier (f * g) = \fourier h$, we conclude that $\Phi \cdot \Psi \in W_\profile^\infty$, as desired.

We now show that $W_\profile^\infty$ contains $\schw$. Indeed: if $\ph \in \schw$ and $\psi = \fourier^{-1} \ph$, then $\psi \in \schw$, and so, in particular, $\psi$ is integrable. Furthermore, by condition~\ref{it:oper:2} in Definition~\ref{def:oper}, $(1 + |x|)^{d + 1} \profile(x) \psi(x)$ is a bounded function of $x \in \R^d$, and so in particular $\profile \psi$ is integrable. Thus, $\ph = \fourier \psi$ indeed belongs to $W_\profile^\infty$.

It remains to observe that if $\ph \in \schw$ and $\Phi \in W_\profile^\infty$, then $\ph \in W_\profile^\infty$, and therefore $\ph \cdot \Phi \in W_\profile^\infty$.
\end{proof}

In order to prove that $W_\profile^\infty$ is a Wiener-type algebra, we only need to verify that $W_\profile^\infty$ satisfies condition~\ref{it:al:3} of Definition~\ref{def:algebra}, that is, a variant of Wiener's $1/f$ theorem holds in $W_\profile^\infty$. While this result is known (see~\cite{bss} for further discussion and references), we provide a complete proof in order to prepare the reader for a similar, but slightly more involved argument in the next section, in the proof of Lemma~\ref{lem:funct:2}.

\begin{lemma}
\label{lem:oper:2}
Let $W_\profile^\infty$ be the set defined in Definition~\ref{def:oper}. Suppose that $K \subseteq \R^d$ is a compact set, $\Phi \in W_\profile^\infty$ and $\Phi(\xi) \ne 0$ for every $\xi \in K$. Then there is $\tilde{\Phi} \in W_\profile^\infty$ such that $\Phi(\xi) \tilde{\Phi}(\xi) = 1$ for every $\xi \in K$.
\end{lemma}

\begin{proof}
We follow the proof of the classical Wiener's $1/f$ lemma given in~\cite{newman}. Let us denote by $\|\cdot\|_p$ the usual norm in $L^p(\R^d)$, where $p \in [1, \infty]$. We begin with the following elementary observation.

Clearly, if $\Phi = \fourier f$ and $f$ is integrable, then
\formula{
 \|\Phi\|_\infty & \le \|f\|_1 .
}
Conversely, for every $k = 1, 2, \ldots$ there is a constant $c_{d, k}$ such that if $\Phi$ is smooth, both $\Phi$ and $\Delta^k \Phi$ are integrable, and $f = \fourier^{-1} \Phi$, then
\formula{
 \sup \{ (1 + |x|^{2 k}) |f(x)| : x \in \R^d \} & \le c_{d, k} \|\Phi + (-\Delta)^k \Phi\|_1 .
}
In particular, if $2 k > d$, we find that $f$ is integrable, while if $2 k > d + q$ with $q$ as in condition~\ref{it:oper:2} in Definition~\ref{def:oper}, then $\profile f$ is integrable. It follows that if $2 k > d + q$, then there is a constant $c_{d, k}$ such that whenever $\Phi$ is smooth, and $\Phi$ and $\Delta^k \Phi$ are integrable, then $\Phi \in W_\profile^\infty$ and
\formula[eq:oper:bound]{
 \|(1 + \profile) f\|_1 & \le c_{d, k} \|\Phi + (-\Delta)^k \Phi\|_1 ,
}
where $f = \fourier^{-1} \Phi$.

We return to the proof of the lemma. Suppose that $\Phi \in W_\profile^\infty$, that is, $\Phi = \fourier f$ for an integrable function $f$ such that also $\profile f$ is integrable. Suppose furthermore that $K$ is a compact set and $\Phi(\xi) \ne 0$ for every $\xi \in K$. Let $\eps > 0$ be small enough, so that $|\Phi(\xi)| > 3 \eps$ for every $\xi$ in some bounded neighbourhood $U$ of $K$. With no loss of generality we assume that $\eps < \tfrac{1}{2}$. Choose $\ph \in \schw$ so that the support of $\ph$ is a compact subset of $U$ and $\ph(\xi) = 1$ for $\xi \in K$. Our goal is to prove that
\formula{
 \tilde{\Phi}(\xi) & = \frac{\ph(\xi)}{\Phi(\xi)}
}
is an element of $W_\profile^\infty$; here, of course, $\tilde{\Phi}(\xi) = 0$ for $\xi \in \R^d \setminus U$. Once this is proved, we have $\Phi(\xi) \tilde{\Phi}(\xi) = \ph(\xi) = 1$ for every $\xi \in K$, and so $\tilde{\Phi}$ has the desired property.

Define $g(x) = f(x)$ when $|x| < r$ and $g(x) = 0$ otherwise, where $r$ is large enough, so that
\formula{
 \|f - g\|_1 + \|\profile f - \profile g\|_1 & < \eps .
}
Clearly, $g$ and $\profile g$ are integrable, and therefore $\Psi = \fourier g$ is an element of $W_\profile^\infty$. Furthermore,
\formula{
 \|\Phi - \Psi\|_\infty & \le \|f - g\|_1 < \eps .
}
In particular, for $\xi \in U$ we have
\formula{
 |\Psi(\xi)| & \ge |\Phi(\xi)| - |\Phi(\xi) - \Psi(\xi)| > 3 \eps - \eps = 2 \eps ,
}
and
\formula{
 \biggl|\frac{\Psi(\xi) - \Phi(\xi)}{\Psi(\xi)}\biggr| & < \frac{\eps}{2 \eps} = \frac{1}{2} \, .
}
It follows that for $\xi \in U$,
\formula{
 \frac{1}{\Phi(\xi)} & = \sum_{n = 0}^\infty \frac{(\Psi(\xi) - \Phi(\xi))^n}{(\Psi(\xi))^{n + 1}} \, ,
}
and therefore for every $\xi \in \R^d$,
\formula[eq:oper:wiener]{
 \tilde{\Phi}(\xi) & = \frac{\ph(\xi)}{\Phi(\xi)} = \sum_{n = 0}^\infty (\Psi(\xi) - \Phi(\xi))^n \, \frac{\ph(\xi)}{(\Psi(\xi))^{n + 1}} \, .
}
We study the terms $(\Psi - \Phi)^n$ and $\ph / \Psi^{n + 1}$ separately.

Recall that $\Psi - \Phi = \fourier(g - f)$ and
\formula{
 \|(1 + \profile) (g - f)\|_1 & = \|g - f\|_1 + \|\profile g - \profile f\|_1 < \eps .
}
If $(g - f)^{*n}$ denotes the $n$-fold convolution of $g - f$, then $\fourier (g - f)^{*n} = (\Psi - \Phi)^n$. Furthermore, using condition~\ref{it:oper:1} in Definition~\ref{def:oper} as in~\eqref{eq:oper:conv}, we find that
\formula[eq:oper:est:1]{
 \bigl\| (1 + \profile) (g - f)^{*n} \bigr\|_1 & \le \bigl( \|(1 + \profile) (g - f)\|_1 \bigr)^n < \eps^n .
}
This is the desired bound for $(\Psi - \Phi)^n$, and we turn to the estimate of $\ph / \Psi^{n + 1}$.

Since $\Psi = \fourier g$ and $g$ has compact support, $\Psi$ is smooth in $\R^d$. Additionally, $|\Psi(\xi)| \ge 2 \eps > 0$ for $\xi \in U$, and $\ph(\xi) = 0$ for $\xi$ outside a compact subset of $U$. It follows that $\ph / \Psi^{n + 1}$ is smooth on $\R^d$, and equal to zero in $\R^d \setminus U$. In particular, by~\eqref{eq:oper:bound}, we find that $\ph / \Psi^{n + 1} \in W_\profile^\infty$, and if $h_n = \fourier^{-1}(\ph / \Psi^{n + 1})$, then
\formula{
 \bigl\|(1 + \profile) h_n\bigr\|_1 & \le c_{d, k} \bigl\|\ph / \Psi^{n + 1} + (-\Delta)^k (\ph / \Psi^{n + 1})\bigr\|_1 ,
}
where $k$ is a fixed sufficiently large positive integer and $c_{d, k}$ is an appropriate constant. By applying the product rule to $(-\Delta)^k (\ph / \Psi^{n + 1})$, we obtain a fixed number of terms, each of which is a product of: the derivative of $\ph$ of some order $j$, where $0 \le j \le 2 k$; a finite number of derivatives of $\Psi$ of total order $2 k - j$; $\Psi^{-n - 1 - i}$, where $0 \le i \le 2 k - j$; and a coefficient, which is an appropriate polynomial of $n$ of degree at most $2 k$. Furthermore, $|\Psi(\xi)| \ge 2 \eps$ for $\xi \in U$, and $2 \eps < 1$. Thus, $\Psi^{-n - 1 - i}(\xi) \le (2 \eps)^{-n - 1 - 2 k}$ when $\xi \in U$ and $0 \le i \le 2 k$. It follows that there is a constant $c_{d, k, \ph, \Psi}$ such that
\formula[eq:oper:est:2]{
 \bigl\|(1 + \profile) h_n\bigr\|_1 & \le c_{d, k, \ph, \Psi} |U| (1 + n^{2 k}) (2 \eps)^{-n - 1 - 2 k} .
}
This is the desired estimate for $\ph / \Psi^{n + 1}$.

We combine~\eqref{eq:oper:est:1} and~\eqref{eq:oper:est:2} as in~\eqref{eq:oper:conv}:
\formula{
 \bigl\|(1 + \profile) (g - f)^{*n} * h_n\bigr\|_1 & \le \bigl\|(1 + \profile) (g - f)^{*n}\bigr\|_1 \cdot \bigl\|(1 + \profile) h_n\bigr\|_1 \\
 & \le c_{d, k, \ph, \Psi} |U| (1 + n^k) (2 \eps)^{-1 - 2 k} \cdot 2^{-n} .
}
By Fubini's theorem, we obtain
\formula{
 \biggl\|(1 + \profile) \sum_{n = 0}^\infty (g - f)^{*n} * h_n\biggr\|_1 & < \infty .
}
In particular, the series
\formula{
 \tilde{f} & = \sum_{n = 0}^\infty (g - f)^{*n} * h_n
}
defines an integrable function such that also $\profile \tilde{f}$ is integrable, and again by Fubini's theorem we find that
\formula{
 \fourier \tilde{f} & = \sum_{n = 0}^\infty \fourier\bigl((g - f)^{*n} * h_n\bigr) \\
 & = \sum_{n = 0}^\infty \fourier\bigl((g - f)^{*n}\bigr) \fourier h_n \\
 & = \sum_{n = 0}^\infty (\Psi - \Phi)^n (\ph / \Psi^{n + 1}) .
}
Thus, by~\eqref{eq:oper:wiener}, $\fourier \tilde{f} = \tilde{\Phi}$, and therefore $\tilde{\Phi} \in W_\profile^\infty$, as desired.
\end{proof}

Lemmas~\ref{lem:oper:1} and~\ref{lem:oper:2} immediately lead to the following result.

\begin{proposition}
\label{prop:oper}
The set $W_\profile^\infty$ defined in Definition~\ref{def:oper} is a Wiener-type algebra.
\end{proposition}

Suppose that $\profile$ satisfies the conditions of Definition~\ref{def:oper}, $\check{\profile}(x) = \profile(-x)$, and $\check{L}$ is an integrable distribution such that $\check{L} \schwconv \check{\profile}$ is well-defined. Then for every $\ph \in \schw$ the integrable function $\check{L} * \ph$ is convolvable with $\check{\profile}$, and so, in particular, $\|(\check{L} * \ph) \cdot \profile\|_1 < \infty$. Thus, $\fourier \ph \cdot \fourier \check{L} = \fourier(\check{L} * \ph) \in W_\profile^\infty$. By choosing $\ph$ such that $\fourier \ph(\xi) = 1$ for $\xi$ in a given compact set, we find that $\fourier \check{L}$ belongs to $W_\profile^\infty$ locally on $\R^d$.

We claim that if $h$ is a function on $\R^d$ such that $h / (1 + \check{\profile})$ is bounded, then $h$ corresponds to a tempered distribution and $\fourier h$ acts on $W_\profile^\infty$. Indeed: by condition~\ref{it:oper:2} in Definition~\ref{def:oper}, $h$ is bounded by some polynomial and hence it corresponds to a tempered distribution. Furthermore, if $\Phi, \Psi \in W_\profile^\infty$, then $\Phi = \fourier f$ and $\Psi = \fourier g$ for some integrable functions $f, g$ such that also $\profile f$ and $\profile g$ are integrable. Thus, by condition~\ref{it:oper:1} in Definition~\ref{def:oper} and Fubini's theorem, we find that
\formula{
 & |h| * |f| * |g|(x) = \int_{\R^d} \int_{\R^d} h(x - y - z) f(y) g(z) dy dz \\
 & \qquad \le (1 + \profile(-x)) \int_{\R^d} \int_{\R^d} \frac{h(x - y - z)}{1 + \profile(-x + y + z)} \, (1 + \profile(y)) f(y) (1 + \profile(z)) g(z) dy dz \\
 & \qquad \le (1 + \check{\profile}(x)) \|h / (1 + \check{\profile})\|_\infty \|(1 + \profile) g\|_1 \|(1 + \profile) h\|_1 .
}
Therefore, by Fubini's theorem, we have
\formula{
 h * (f * g) & = (h * f) * g ,
}
and in fact, due to condition~\ref{it:oper:2} in Definition~\ref{def:oper},
\formula{
 h \schwconv (f * g) & = (h \schwconv f) \schwconv g .
}
Applying the exchange formula, we conclude that
\formula{
 \fourier H \schwprod (\Phi \cdot \Psi) = (\fourier H \schwprod \Phi) \schwprod \Psi ,
}
which completes the proof of our claim.

As an immediate corollary of Theorem~\ref{thm:liouville} and Proposition~\ref{prop:oper}, as well as the two properties discussed above, and after exchanging the roles of $\profile$ and $\check{\profile}$, we obtain the following variant of Liouville's theorem.

\begin{corollary}[Liouville's theorem under generalised moment condition]
\label{cor:liouville:oper}
Suppose that $\profile$ satisfies the conditions of Definition~\ref{def:oper}. Let $\check{L}$ be an integrable distribution which is convolvable with $\profile$. Suppose that $\fourier \check{L}$, which is necessarily a continuous function, has no zeroes in an open set $U$. Let $h$ be a function such that $h / (1 + \profile)$ is bounded on $\R^d$ and $\check{L} \schwconv h = 0$. Then the spectrum of $h$ is contained in $\R^d \setminus U$.

In particular, if $U = \R^d \setminus \{0\}$, then $h$ is a polynomial.
\end{corollary}

We conclude this section with examples of admissible functions $\profile$.

\begin{example}
\label{ex:liouville:oper:power}
If $\alpha \ge 0$, then $\profile(x) = (1 + |x|)^\alpha - 1$ satisfies conditions~\ref{it:oper:1} and~\ref{it:oper:2} in Definition~\ref{def:oper}. Indeed,
\formula{
 (1 + \profile(x)) (1 + \profile(y)) & = (1 + |x|)^\alpha (1 + |y|)^\alpha \\
 & \ge (1 + |x| + |y|)^\alpha \ge (1 + |x + y|)^\alpha = 1 + \profile(x + y) .
}
Thus, if $L$ is a Lévy operator which is not concentrated on a proper closed subgroup of $\R^d$, and the Lévy measure $\nu$ of $L$ satisfies
\formula{
 \int_{\R^d \setminus B} |y|^\alpha \nu(dy) & < \infty ,
}
then every $L$-harmonic function $h$ (in the sense of tempered distributions) such that $(1 + |x|)^{-\alpha} h(x)$ is a bounded function of $x \in \R^d$ is necessarily a polynomial. When $\alpha < 1$, then it follows that $h$ is in fact constant. For $\alpha = 0$ we thus recover Liouville's theorem for bounded $L$-harmonic functions given in Corollary~\ref{cor:liouville:bounded}.
\end{example}

\begin{example}
\label{ex:liouville:oper:log}
If $\beta \ge 0$, then it is easy to see that $\profile(x) = (\log(e^2 + |x|))^\beta$ satisfies conditions~\ref{it:oper:1} and~\ref{it:oper:2} in Definition~\ref{def:oper}. Indeed: we have
\formula{
 \profile(x) \profile(y) & = \bigl( \log(e^2 + |x|) \log(e^2 + |y|) \bigr)^\beta \\
 & \ge \bigl( \log(e^2) \log(e^2 + \max\{|x|, |y|\}) \bigr)^\beta \\
 & \ge \bigl( 2 \log(e^2 + \tfrac{1}{2} (|x| + |y|)) \bigr)^\beta \\
 & = \bigl( \log(e^4 + e^2 (|x| + |y|) + \tfrac{1}{4} (|x| + |y|)^2) \bigr)^\beta \\
 & \ge \bigl( \log(e^2 + |x + y|) \bigr)^\beta = \profile(x + y) ,
}
and hence $1 + \profile(x + y) \le 1 + \profile(x) \profile(y) \le (1 + \profile(x)) (1 + \profile(y))$.

It follows that if $L$ is a Lévy operator which is not concentrated on a proper closed subgroup of $\R^d$, and the Lévy measure $\nu$ of $L$ satisfies
\formula{
 \int_{\R^d \setminus B} (\log |y|)^\beta \nu(dy) & < \infty ,
}
then every $L$-harmonic function $h$ (in the sense of tempered distributions) such that $(\log(e + |x|))^{-\beta} h(x)$ is a bounded function of $x \in \R^d$ is necessarily constant.
\end{example}

%                            ---------- o ----------

\subsection{Harmonic functions with finite generalised negative moments}
\label{sec:signed:funct}

As in the previous section, we define a class of Wiener-type algebras, again parameterised by auxiliary functions $\profile$. Once again we discuss examples of admissible functions $\profile$ at the end of this section, and here we remark that a typical example is $\profile(x) = c_{d, \alpha} (1 + |x|)^{-d - \alpha}$, where $\alpha > 0$ is a parameter and $c_{d, \alpha}$ is a sufficiently small constant.

\begin{definition}
\label{def:funct}
Suppose that $\profile$ is an integrable function on $\R^d$ with the following properties:
\begin{enumerate}[label=(\alph*)]
\item\label{it:funct:1} we have $\profile * \profile(x) \le \profile(x)$ for every $x \in \R^d$, and if $\profile_r(x) = \profile(x)$ when $|x| \ge r$ and $\profile_r(x) = 0$ otherwise, then
\formula{
 \lim_{r \to \infty} \sup \biggl\{\frac{\profile_r * \profile_r(x)}{\profile(x)} : x \in \R^d \biggr\} & = 0 ;
}
\item\label{it:funct:2} for some positive constants $c, q$ we have $\profile(x) \ge c (1 + |x|)^{-q}$ for every $x \in \R^d$.
\end{enumerate}
Define $W_\profile^1$ to be the class of Fourier transforms of integrable functions $f$ such that $f / \profile$ is bounded.
\end{definition}

We remark that if $\profile * \profile \le c \profile$ for some constant $c$, then we may replace $\profile$ by $c^{-1} \profile$ to get a function which satisfies $\profile * \profile \le \profile$. We also note that condition $\profile * \profile(x) \le c \profile(x)$ is called \emph{direct jump property} in~\cite{ks}, and condition~\ref{it:funct:1} plays a crucial role in~\cite{kp}; we refer to these papers for further discussion and references.

Below we prove that $W_\profile^1$ is a Wiener-type algebra according to Definition~\ref{def:algebra}. The argument is similar to the one applied in the previous section. Nevertheless, since there are essential differences between the two, we provide all details.

\begin{lemma}
\label{lem:funct:1}
The set $W_\profile^1$ defined in Definition~\ref{def:funct} is an algebra of continuous functions which satisfies conditions~\ref{it:al:1} and~\ref{it:al:2} of Definition~\ref{def:algebra}.
\end{lemma}

\begin{proof}
Clearly, $W_\profile^1$ is a linear space. If $\Phi \in W_\profile^1$, then $\Phi$ is the Fourier transform of an integrable function $f$, and hence $\Phi$ is a bounded continuous function. Thus, $W_\profile^1$ is a class of continuous functions, and every $\Phi \in W_\profile^1$ corresponds to a tempered distribution.

We claim that $W_\profile^1$ is indeed an algebra of functions. If $\Phi, \Psi \in W_\profile^1$, then $\Phi = \fourier f$ and $\Psi = \fourier g$ for some integrable functions $f, g$, and for some constants $c_1, c_2$ we have $|f(x)| \le c_1 \profile(x)$ and $|g(x)| \le c_2 \profile(x)$ for every $x \in \R^d$. It follows that $h = f * g$ is an integrable function, and
\formula[eq:funct:conv]{
 |h(x)| & = |f * g(x)| \le |f| * |g|(x) \le (c_1 \profile) * (c_2 \profile)(x) = c_1 c_2 \profile * \profile(x) \le c_1 c_2 \profile(x) ,
}
that is, $h / \profile$ is bounded. Since $\Phi \cdot \Psi = \fourier f \cdot \fourier g = \fourier (f * g) = \fourier h$, we conclude that $\Phi \cdot \Psi \in W_\profile^1$, as desired.

We now show that $W_\profile^1$ contains $\schw$. Indeed: if $\ph \in \schw$ and $\psi = \fourier^{-1} \ph$, then $\psi \in \schw$, and so, in particular, $\psi$ is integrable. Furthermore, by condition~\ref{it:funct:2} in Definition~\ref{def:funct}, $\psi / \profile$ is bounded. Thus, $\ph = \fourier \psi$ indeed belongs to $W_\profile^1$.

It remains to observe that if $\ph \in \schw$ and $\Phi \in W_\profile^1$, then $\ph \in W_\profile^1$, and therefore $\ph \cdot \Phi \in W_\profile^1$.
\end{proof}

As before, in order to prove that $W_\profile^1$ is a Wiener-type algebra, it remains to verify that $W_\profile^1$ satisfies condition~\ref{it:al:3} of Definition~\ref{def:algebra}, that is, a variant of Wiener's $1/f$ theorem holds in $W_\profile^1$.

\begin{lemma}
\label{lem:funct:2}
Let $W_\profile^1$ be the set defined in Definition~\ref{def:funct}. Suppose that $K \subseteq \R^d$ is a compact set, $\Phi \in W_\profile^1$ and $\Phi(\xi) \ne 0$ for every $\xi \in K$. Then there is $\tilde{\Phi} \in W_\profile^1$ such that $\Phi(\xi) \tilde{\Phi}(\xi) = 1$ for every $\xi \in K$.
\end{lemma}

\begin{proof}
Once again we follow the proof of the classical Wiener's $1/f$ lemma given in~\cite{newman}. We denote by $\|\cdot\|_p$ the usual norm in $L^p(\R^d)$, where $p \in [1, \infty]$. As in the proof of Lemma~\ref{lem:oper:2}, we have the following two observations.

If $\Phi = \fourier f$ and $f$ is integrable, then
\formula{
 \|\Phi\|_\infty & \le \|f\|_1 .
}
Conversely, if $k = 1, 2, \ldots$, $\Phi$ is smooth, both $\Phi$ and $\Delta^k \Phi$ are integrable, and $f = \fourier^{-1} \Phi$, then
\formula{
 \sup \{ (1 + |x|^{2 k}) |f(x)| : x \in \R^d \} & \le c_{d, k} \|\Phi + (-\Delta)^k \Phi\|_1 ,
}
where $c_{d, k}$ is an appropriate constant. In particular, if $2 k > d$, then $f$ is integrable, while if $2 k \ge q$ with $q$ as in condition~\ref{it:funct:2} in Definition~\ref{def:funct}, then $f / \profile$ is bounded. It follows that if $2 k > d$ and $2 k \ge q$, then there is a constant $c_{d, k}$ such that whenever $\Phi$ is smooth, and $\Phi$ and $\Delta^k \Phi$ are integrable, then $\Phi \in W_\profile^1$ and
\formula[eq:funct:bound]{
 \|f\|_1 + \|f / \profile\|_\infty & \le c_{d, k} \|\Phi + (-\Delta)^k \Phi\|_1 ,
}
where $f = \fourier^{-1} \Phi$.

We return to the proof of the lemma. Suppose that $\Phi \in W_\profile^1$, that is, $\Phi = \fourier f$ for an integrable function $f$ such that $f / \profile$ is bounded. Suppose furthermore that $K$ is a compact set and $\Phi(\xi) \ne 0$ for every $\xi \in K$. Let $\eps > 0$ be small enough, so that $|\Phi(\xi)| > 3 \eps$ for every $\xi$ in some bounded neighbourhood $U$ of $K$. Choose $\ph \in \schw$ so that the support of $\ph$ is a compact subset of $U$ and $\ph(\xi) = 1$ for $\xi \in K$. As in the proof of Lemma~\ref{lem:oper:2}, our goal is to prove that
\formula{
 \tilde{\Phi}(\xi) & = \frac{\ph(\xi)}{\Phi(\xi)}
}
is an element of $W_\profile^1$; here, of course, $\tilde{\Phi}(\xi) = 0$ for $\xi \in \R^d \setminus U$. Once this is shown, we have $\Phi(\xi) \tilde{\Phi}(\xi) = \ph(\xi) = 1$ for every $\xi \in K$, and so $\tilde{\Phi}$ has the desired property.

For $r > 0$, let $B_r$ denote the ball of radius $r$ centred at the origin, and denote $\profile_r(x) = \profile(x) \ind_{B_r}(x)$, as in condition~\ref{it:funct:1} in Definition~\ref{def:funct}. Choose $r$ large enough, so that if $g(x) = f(x) \ind_{B_r}(x)$, then
\formula{
 \|f - g\|_1 & < \eps
}
and
\formula[eq:funct:r]{
 \profile_r * \profile_r(x) & \le \frac{\eps^2}{\|f / \profile\|_\infty^2} \, \profile(r)
}
for every $x \in \R^d$ (see condition~\ref{it:funct:1} in Definition~\ref{def:funct}). Clearly, $g$ is integrable and $g / \profile$ is bounded, and therefore $\Psi = \fourier g$ is an element of $W_\profile^1$. Furthermore,
\formula{
 \|\Phi - \Psi\|_\infty & \le \|f - g\|_1 < \eps .
}
In particular, for $\xi \in U$ we have
\formula{
 |\Psi(\xi)| & \ge |\Phi(\xi)| - |\Phi(\xi) - \Psi(\xi)| > 3 \eps - \eps = 2 \eps ,
}
and
\formula{
 \biggl|\frac{\Psi(\xi) - \Phi(\xi)}{\Psi(\xi)}\biggr| & < \frac{\eps}{2 \eps} = \frac{1}{2} \, ,
}
so that for $\xi \in U$ we have
\formula{
 \frac{1}{\Phi(\xi)} & = \frac{1}{\Psi(\xi)} \, \sum_{n = 0}^\infty \biggl(\frac{\Psi(\xi) - \Phi(\xi)}{\Psi(\xi)}\biggr)^n .
}
We conclude that for every $\xi \in \R^d$,
\formula[eq:funct:wiener]{
 \tilde{\Phi}(\xi) & = \frac{\ph(\xi)}{\Phi(\xi)} = \sum_{n = 0}^\infty \frac{\ph(\xi)}{(\Psi(\xi))^{n + 1}} \, (\Psi(\xi) - \Phi(\xi))^n ,
}
and we study the terms $\ph / \Psi^{n + 1}$ and $(\Psi - \Phi)^n$ separately.

Recall that $\Psi - \Phi = \fourier(g - f)$ and $\|g - f\|_1 < \eps$. Furthermore, if $\lambda = \|f / \profile\|_\infty$, then
\formula{
 |g(x) - f(x)| & \le \lambda \profile_r(x)
}
for every $x \in \R^d$. If $(g - f)^{*n}$ denotes the $n$-fold convolution of $g - f$, then $\fourier (g - f)^{*n} = (\Psi - \Phi)^n$, and
\formula[eq:funct:est:1]{
 \|(g - f)^{*n}\|_1 & \le \|g - f\|_1^n < \eps^n .
}
The estimate of $(g - f)^{*n} / \profile$ is more involved. If $n$ is even, then, by~\eqref{eq:funct:r},
\formula{
 |(g - f)^{*n}(x)| & \le (\lambda \profile_r)^{*n}(x) \\
 & \le \lambda^n (\profile_r * \profile_r)^{*n / 2}(x) \\
 & \le \lambda^n (\lambda^{-2} \eps^2 \profile)^{*n / 2}(x) \\
 & = \eps^n \profile^{*n / 2}(x) \le \eps^n \profile(x) ,
}
and if $n$ is odd, then, in a similar manner,
\formula{
 |(g - f)^{*n}(x)| & \le (\lambda \profile_r)^{*n}(x) \\
 & \le \lambda^n (\profile_r * \profile_r)^{*\lfloor n/2\rfloor} * \profile_r(x) \\
 & \le \lambda^n (\lambda^{-2} \eps^2 \profile)^{*\lfloor n/2\rfloor} * \profile(x) \\
 & = \lambda \eps^{n - 1} \profile^{*\lfloor n/2\rfloor + 1}(x) \le \lambda \eps^{n - 1} \profile(x) .
}
Thus, for an arbitrary $n$,
\formula[eq:funct:est:2]{
 |(g - f)^{*n}(x)| & \le (\lambda + \eps) \eps^{n - 1} \profile(x) .
}
These are the desired bounds for $(\Psi - \Phi)^n$, and we now study $\ph / \Psi^{n + 1}$.

Since $\Psi = \fourier g$ and $g$ has compact support, $\Psi$ is smooth in $\R^d$. Additionally, $|\Psi(\xi)| \ge 2 \eps > 0$ for $\xi \in U$, and $\ph(\xi) = 0$ for $\xi$ outside a compact subset of $U$. It follows that $\ph / \Psi^{n + 1}$ is smooth on $\R^d$, and equal to zero in $\R^d \setminus U$. In particular, by~\eqref{eq:funct:bound}, we find that $\ph / \Psi^{n + 1} \in W_\profile^1$, and if $h_n = \fourier^{-1}(\ph / \Psi^{n + 1})$, then
\formula{
 \|h_n\|_1 + \|h_n / \profile\|_\infty & \le c_{d, k} \bigl\|\ph / \Psi^{n + 1} + (-\Delta)^k (\ph / \Psi^{n + 1})\bigr\|_1 ,
}
where $k$ is a fixed sufficiently large positive integer and $c_{d, k}$ is an appropriate constant. By applying the product rule to $(-\Delta)^k (\ph / \Psi^{n + 1})$, as in the proof of Lemma~\ref{lem:oper:2} we find that there is a constant $c_{d, k, \ph, \Psi}$ such that
\formula[eq:funct:est:3]{
 \|h_n\|_1 + \|h_n / \profile\|_\infty & \le c_{d, k, \ph, \Psi} |U| (1 + n^k) (2 \eps)^{-n - 1 - 2 k} ,
}
and this is the desired estimate for $\ph / \Psi^{n + 1}$.

We combine~\eqref{eq:funct:est:1} and~\eqref{eq:funct:est:3} to find that
\formula{
 \|(g - f)^{*n} * h_n\|_1 & \le \|(g - f)^{*n}\|_1 \|h_n\|_1 \le \cdot c_{d, k, \ph, \Psi} |U| (1 + n^k) (2 \eps)^{-1 - 2 k} \cdot 2^{-n} ,
}
and similarly we combine combine~\eqref{eq:funct:est:2} and~\eqref{eq:funct:est:3} as in~\eqref{eq:funct:conv} to find that
\formula{
 |(g - f)^{*n} * h_n(x)| & \le c_{d, k, \ph, \Psi} |U| (1 + n^k) (2 \eps)^{-2 - 2 k} (\lambda + \eps) \cdot 2^{1 - n} \profile(x) .
}
By Fubini's theorem, we obtain
\formula{
 \biggl\|\sum_{n = 0}^\infty (g - f)^{*n} * h_n\biggr\|_1 & < \infty ,
}
and
\formula{
 \biggl\| \frac{1}{\profile} \sum_{n = 0}^\infty (g - f)^{*n} * h_n \biggr\|_\infty & < \infty .
}
In particular, the series
\formula{
 \tilde{f} & = \sum_{n = 0}^\infty (g - f)^{*n} * h_n
}
defines an integrable function such that $\tilde{f} / \profile$ is bounded, and as in the proof of Lemma~\ref{lem:oper:2}, by Fubini's theorem we find that
\formula{
 \fourier \tilde{f} & = \sum_{n = 0}^\infty (\Psi - \Phi)^n (\ph / \Psi^{n + 1}) .
}
Thus, by~\eqref{eq:funct:wiener}, $\fourier \tilde{f} = \tilde{\Phi}$, and therefore $\tilde{\Phi} \in W_\profile^1$, as desired.
\end{proof}

As an immediate corollary of Lemmas~\ref{lem:funct:1} and~\ref{lem:funct:2}, we obtain the following result.

\begin{proposition}
\label{prop:funct}
The set $W_\profile^1$ defined in Definition~\ref{def:funct} is a Wiener-type algebra.
\end{proposition}

Suppose that $\profile$ satisfies the conditions of Definition~\ref{def:funct}, and $\check{L}$ is an integrable distribution such that for every $\ph \in \schw$ the Fourier transform of $(\check{L} * \ph) / \profile$ is a bounded function. Then, by definition, $\fourier \ph \cdot \fourier\check{L} = \fourier(\check{L} * \ph) \in W_\profile^1$. By choosing $\ph$ such that $\fourier \ph(\xi) = 1$ for $\xi$ in a given compact set, we find that $\fourier \check{L}$ belongs to $W_\profile^1$ locally on $\R^d$.

Let $\check{\profile}(x) = \profile(-x)$. We claim that if $h$ is a function on $\R^d$ such that $\check{\profile} h$ is integrable, then $h$ corresponds to a tempered distribution and $\fourier h$ acts on $W_\profile^1$. Indeed: by condition~\ref{it:funct:2} in Definition~\ref{def:funct}, $h$ is bounded by the product of an integrable function $\check{\profile} h$ and a polynomial, and hence it corresponds to a tempered distribution. Using both conditions in Definition~\ref{def:funct}, we find that for some constants $c, q$ we have $\profile(x - y) \le c (1 + |x|)^q \profile(-y)$ for every $x, y \in \R^d$. It follows that $h$ and $\profile$ are convolvable, and for some constants $c, q$ we have
\formula{
 |h| * \profile(x) & \le c |h| * \profile(0) (1 + |x|)^q .
}
Finally, if $\Phi, \Psi \in W_\profile^1$, then $\Phi = \fourier f$ and $\Psi = \fourier g$ for some integrable functions $f, g$ such that $f / \profile$ and $g / \profile$ are bounded. Thus, by condition~\ref{it:funct:1} in Definition~\ref{def:funct} and Fubini's theorem, we find that
\formula*[eq:funct:fubini]{
 |h| * |f| * |g|(x) & \le \|f / \profile\|_\infty \|g / \profile\|_\infty |h| * \profile * \profile(x) \\
 & \le \|f / \profile\|_\infty \|g / \profile\|_\infty |h| * \profile(x) \\
 & \le c \|f / \profile\|_\infty \|g / \profile\|_\infty |h| * \profile(0) (1 + |x|)^q
}
for every $x \in \R^d$. Therefore, by Fubini's theorem, we have
\formula{
 h * (f * g) & = (h * f) * g ,
}
and in fact, by~\eqref{eq:funct:fubini},
\formula{
 h \schwconv (f * g) & = (h \schwconv f) \schwconv g .
}
Applying the exchange formula, we conclude that
\formula{
 \fourier H \schwprod (\Phi \cdot \Psi) = (\fourier H \schwprod \Phi) \schwprod \Psi ,
}
which completes the proof of our claim.

After exchanging the roles of $\profile$ and $\check{\profile}$, Theorem~\ref{thm:liouville} and Proposition~\ref{prop:funct}, as well as the two properties discussed above, immediately lead to the following variant of Liouville's theorem.

\begin{corollary}[Liouville's theorem under generalised negative moment condition]
\label{cor:liouville:funct}
Suppose that $\profile$ satisfies the conditions of Definition~\ref{def:funct}. Let $\check{L}$ be an integrable distribution such that for every $\ph \in \schw$ there is a constant $c$ such that $|\check{L} * \ph(x)| \le c \profile(-x)$ for every $x \in \R^d$. Suppose that $\fourier \check{L}$, which is necessarily a continuous function, has no zeroes in an open set $U$. Let $h$ be a function such that $\profile h$ is integrable on $\R^d$ and $\check{L} \schwconv h = 0$. Then the spectrum of $h$ is contained in $\R^d \setminus U$.

In particular, if $U = \R^d \setminus \{0\}$, then $h$ is a polynomial.
\end{corollary}

We conclude this section with examples of admissible functions $\profile$. For this, we need the following auxiliary result, which resembles Theorem~1.1 in~\cite{bgo}.

\begin{lemma}
\label{lem:liouville:funct:radial}
Suppose that $\ph$ is a positive, decreasing, continuous function on $[0, \infty)$ such that $r^{d - 1} \ph(r)$ is integrable with respect to $r \in [0, \infty)$ and
\formula[eq:liouville:funct:radial:reg]{
 \liminf_{r \to \infty} \frac{\ph(2 r)}{\ph(r)} & > 0 .
}
Then, for $\eps > 0$ small enough, the function $\profile(x) = \eps \ph(|x|)$ satisfies all conditions of Definition~\ref{def:funct}.
\end{lemma}

\begin{proof}
Clearly, $\profile$ is integrable over $\R^d$. We need to verify conditions~\ref{it:funct:1} and~\ref{it:funct:2} in Definition~\ref{def:funct}.

Since $\ph$ is decreasing, we have
\formula{
 \min\bigl\{\ph(|x - y|), \ph(|y|)\bigr\} & = \ph\bigl(\max\{|x - y|, |y|\}\bigr) \le \ph(\tfrac{1}{2} |x|) .
}
Hence,
\formula{
 \ph(|x - y|) \ph(|y|) & = \min\bigl\{\ph(|x - y|), \ph(|y|)\bigr\} \max\bigl\{\ph(|x - y|), \ph(|y|)\bigr\} \\
 & \le \ph(\tfrac{1}{2} |x|) \bigl(\ph(|x - y|) + \ph(|y|)\bigr) .
}
It follows that
\formula{
 \profile * \profile(x) & = \eps^2 \int_{\R^d} \ph(|x - y|) \ph(|y|) dy \\
 & \le \eps^2 \ph(\tfrac{1}{2} |x|) \int_{\R^d} \bigl(\ph(|x - y|) + \ph(|y|)\bigr) dy = 2 \eps^2 c_1 \ph(\tfrac{1}{2} |x|) ,
}
where $c_1$ is the integral of $\ph(|x|)$ over $x \in \R^d$. By assumption~\eqref{eq:liouville:funct:radial:reg}, there is a constant $c_2$ such that $\ph(\tfrac{1}{2} |x|) \le c_2 \ph(|x|)$ for every $x \in \R^d$. It follows that
\formula{
 \profile * \profile(x) & \le 2 c_1 c_2 \eps \profile(x) .
}
Hence, if $\eps$ is small enough, we have $\profile * \profile(x) \le \profile(x)$, as desired.

Let $\profile_r$ be defined as in condition~\ref{it:funct:1} in Definition~\ref{def:funct}, and let $B_r$ denote the centred ball of radius $r$. By the same argument as above, we have
\formula{
 \profile_r * \profile_r(x) & = \eps^2 \int_{\R^d} \ind_{\R^d \setminus B_r}(x - y) \ph(|x - y|) \ind_{\R^d \setminus B_r}(y) \ph(|y|) dy \\
 & \le \eps^2 \ph(\tfrac{1}{2} |x|)\int_{\R^d} \ind_{\R^d \setminus B_r}(x - y) \ind_{\R^d \setminus B_r}(y) \bigl( \ph(|x - y|) + \ph(|y|) \bigr) dy \\
 & \le c_2 \eps^2 \ph(|x|)\int_{\R^d} \bigl( \ind_{\R^d \setminus B_r}(x - y) \ph(|x - y|) + \ind_{\R^d \setminus B_r}(y) \ph(|y|) \bigr) dy \\
 & \le 2 c_2 \eps \profile(x) \int_{\R^d} \ind_{\R^d \setminus B_r}(y) \ph(|y|) dy .
}
Thus,
\formula{
 \lim_{r \to \infty} \sup \biggl\{\frac{\profile_r * \profile_r(x)}{\profile(x)} : x \in \R^d \biggr\} & \le \lim_{r \to \infty} \biggl(2 c_2 \eps \int_{\R^d} \ind_{\R^d \setminus B_r}(y) \ph(|y|) dy\biggr) = 0 .
}
We have thus proved that $\profile$ indeed satisfies condition~\ref{it:funct:1} in Definition~\ref{def:funct}.

Condition~\ref{it:funct:2} in Definition~\ref{def:funct} is related to general theory of regular variation, but in fact it is an elementary result: if, as above, $\ph(\tfrac{1}{2} |x|) \le c_2 \ph(|x|)$ for every $x \in \R^d$ and $c_3$ is the infimum of $\ph$ over $[0, 1]$, then for $x \in \R^d \setminus \{0\}$ such that $2^{n - 1} < |x| \le 2^n$, where $n = 0, 1, 2, \ldots$\,, we have
\formula{
 c_3 & \le \ph(2^{-n} |x|) \le c_2^n \ph(|x|) = c_2 (2^{n - 1})^{\log c_2 / \log 2} \ph(|x|) \le c_2 |x|^{\log c_2 / \log 2} \ph(|x|) ,
}
and hence $\ph(|x|) \ge (c_3 / c_2) (1 + |x|)^{-\log c_2 / \log 2}$ for every $x \in \R^d$, as desired.
\end{proof}

\begin{example}
\label{ex:liouville:funct:power}
Using Lemma~\ref{lem:liouville:funct:radial} it is straightforward to verify that if $\alpha > 0$, then $\profile(x) = c_{d, \alpha} (1 + |x|)^{-d - \alpha}$ satisfies the conditions of Definition~\ref{def:funct}. Thus, we get the following result.

Let $B$ denote the unit ball in $\R^d$. Suppose that $L$ is a Lévy operator which is not concentrated on a proper closed subgroup of $\R^d$, and for some constant $c$ the Lévy measure $\nu$ of $L$ satisfies
\formula{
 \nu(x + B) & \le c |x|^{-d - \alpha}
}
for every $x \in \R^d$ such that $|x| \ge 2$. Then every $L$-harmonic function $h$ such that $(1 + |x|)^{-d - \alpha} h(x)$ is an integrable function of $x \in \R^d$ is necessarily a polynomial.

Suppose now that $\nu$ has a density function which is comparable with $|x|^{-d - \alpha}$ for $|x|$ large enough. In this case our integrability condition on $h$, namely, that $(1 + |x|)^{-d - \alpha} h(x)$ is an integrable function, is automatically satisfied by every $L$-harmonic function in the weak sense (as defined in Definition~\ref{def:harmonic}). Thus, by Lemma~\ref{lem:harmonic:schw}, the result given above characterises all functions $h$ which are $L$-harmonic in the weak sense, with no additional conditions on $h$.

The above result is a variant of Theorem~1.4 in~\cite{fw}, while the more general result given in Corollary~\ref{cor:liouville:funct} with this choice of $\profile$ is a minor extension of Theorem~1.1 in~\cite{fw}.
\end{example}

\begin{example}
\label{ex:liouville:funct:power:log}
More generally, by Lemma~\ref{lem:liouville:funct:radial}, we find that if $\profile(x) = \eps \ph(|x|)$ for a positive, decreasing, continuous function $\ph$ on $[0, \infty)$ such that $\ph(|x|)$ is integrable with respect to $x \in \R^d$, and such that $\ph(r)$ is regularly varying as $r \to \infty$, then $\profile$ satisfies the conditions of Definition~\ref{def:funct} if $\eps$ is small enough. In particular, this implies the following result.

Let $\alpha > 0$ and $\beta \in \R$, or $\alpha > 0$ and $\beta > 1$. Let $L$ be a Lévy operator which is not concentrated on a proper closed subgroup of $\R^d$, and such that for some constant $c$ the Lévy measure $\nu$ of $L$ satisfies
\formula{
 \nu(x + B) & \le c |x|^{-d - \alpha} (\log |x|)^{-\beta}
}
for every $x \in \R^d$ such that $|x| \ge 2$ (where $B$ is the unit ball). Then every $L$-harmonic function $h$ such that $(1 + |x|)^{-d - \alpha} (\log(e + |x|))^{-\beta} h(x)$ is an integrable function of $x \in \R^d$ is necessarily a polynomial.

As in the previous example, if $\nu$ has a density function which is comparable with $|x|^{-d - \alpha} (\log |x|)^{-\beta}$ for $|x|$ large enough, then $L$-harmonic functions in the weak sense (see Definition~\ref{def:harmonic}) automatically satisfy the integrability condition, and hence, by Lemma~\ref{lem:harmonic:schw}, the above result describes all $L$-harmonic functions in the weak sense.
\end{example}

\begin{example}
\label{ex:liouville:funct:tensor}
Suppose that for $j = 1, 2, \ldots, k$ a function $\profile_j$ on $\R^{d_j}$ satisfies the conditions of Definition~\ref{def:funct}, and define $d = d_1 + d_2 + \ldots + d_k$ and
\formula{
 \profile(x) & = \prod_{j = 1}^k \profile_j(x_j) ,
}
where $x = (x_1, x_2, \ldots, x_k) \in \R^d$ with $x_j \in \R^{d_j}$. It is then immediate to check that $\profile$ satisfies the conditions of Definition~\ref{def:funct}. Combining this with Example~\ref{ex:liouville:funct:power}, we arrive at the following result.

Let us identify $\R^d$ with $\R^{d_1} \times \R^{d_2} \times \ldots \times \R^{d_k}$ by writing, as above, $x = (x_1, x_2, \ldots, x_k)$, where $x \in \R^d$ and $x_j \in \R^{d_j}$. Suppose that $L = L_1 + L_2 + \ldots + L_k$, where for every $j = 1, 2, \ldots, k$ the operator $L_j$ is a $d_j$-dimensional Lévy operator which is not concentrated on a proper closed subgroup of $\R^{d_j}$, acting in variable $x_j$. Furthermore, suppose that for some constants $c_j$ and $\alpha_j > 0$ the Lévy measure $\nu_j$ of $L_j$ satisfies
\formula{
 \nu_j(x_j + B_j) & \le c_j |x_j|^{-d_j - \alpha_j}
}
for every $x_j \in \R^{d_j}$ such that $|x_j| \ge 2$; here $B_j$ is the unit ball in $\R^{d_j}$. Then every $L$-harmonic function $h$ such that
\formula{
 \biggl(\prod_{j = 1}^d \frac{1}{(1 + |x_j|)^{d_j + \alpha_j}}\biggr) h(x)
}
is an integrable function of $x \in \R^d$ is necessarily a polynomial.

We remark that, unlike in the previous examples, the integrability condition on $h$ does not seem to follow automatically from the definition of an $L$-harmonic function in the weak sense.
\end{example}

%
%                            ---------- o ----------
%

\medskip

\subsection*{Conflict of interest statement}

On behalf of all authors, the corresponding author states that there is no conflict of interest.

\subsection*{Data availability statement}

Data sharing not applicable to this article as no datasets were generated or analysed during the current study.

%
%                            ---------- o ----------
%

%
%                            ---------- o ----------
%

\end{document}